\documentclass [11pt,reqno]{amsart}
\usepackage{amsmath,amssymb,verbatim,geometry,color}
\usepackage[all]{xy}
\usepackage{mathrsfs}
\usepackage{hyperref}
\usepackage[pdftex]{graphicx}
\usepackage{mathtools}
\usepackage[style=alphabetic, maxnames=99, maxalphanames=10, giveninits=true, isbn=false, doi=false, url=false]{biblatex}
\usepackage{quiver}

\def\XXint#1#2#3{{\setbox0=\hbox{$#1{#2#3}{\int}$ }
\vcenter{\hbox{$#2#3$ }}\kern-.6\wd0}}

\renewbibmacro{in:}{}

\newcommand{\A}{\mathbb{A}}

\newcommand{\C}{\mathbb{C}}
\newcommand{\G}{\mathbb{G}}

\newcommand{\N}{\mathbb{N}}
\renewcommand{\P}{\mathbb{P}}
 \newcommand{\Q}{\mathbb{Q}}
 \newcommand{\R}{\mathbb{R}}
 \newcommand{\Z}{\mathbb{Z}}
\newcommand{\T}{\mathbb{T}}

\newcommand{\fa}{\mathfrak{a}}
\newcommand{\fb}{\mathfrak{b}}

\newcommand{\fm}{\mathfrak{m}}

\newcommand{\ft}{\mathfrak{t}}

\renewcommand{\hom}{\mathrm{hom}}

\newcommand{\cF}{\mathcal{F}}

\newcommand{\cO}{\mathcal{O}}

\newcommand{\cR}{\mathcal{R}}

\newcommand{\Xdiv}{X^{\mathrm{div}}}

\newcommand{\Xqm}{X^{\mathrm{qm}}}

\newcommand{\Xval}{X^{\mathrm{val}}}

\renewcommand{\a}{\alpha}
\renewcommand{\b}{\beta}
\renewcommand{\d}{\delta}
\newcommand{\e}{\varepsilon}

\newcommand{\la}{\lambda}

\newcommand{\Om}{\Omega}

\renewcommand{\=}{\coloneqq}

\newcommand{\an}{\mathrm{an}}

\newcommand{\lc}{\mathrm{LC}}

\DeclareMathOperator{\FS}{FS}

\DeclareMathOperator{\Hom}{Hom}
\DeclareMathOperator{\lct}{lct}

\DeclareMathOperator{\Spec}{Spec}

\DeclareMathOperator{\supp}{supp}
\DeclareMathOperator{\vol}{vol}

\DeclareMathOperator{\im}{Im}

\DeclareMathOperator{\ord}{ord}

\DeclareMathOperator{\Proj}{Proj}

\DeclareMathOperator{\gr}{gr}

\DeclareMathOperator{\redu}{red}

\DeclareMathOperator{\Hnot}{H^0}

\DeclareMathOperator{\de}{d}

\DeclareMathOperator{\ratrk}{rat.rk} 
\DeclareMathOperator{\trdeg}{tr.deg}

\DeclareMathOperator{\Val}{Val}

\DeclareMathOperator{\QM}{QM}

\renewcommand{\div}{\mathrm{div}}

\newcommand{\val}{\mathrm{val}}

\newcommand{\na}{\mathrm{na}}

\numberwithin{equation}{section}       

\newtheorem{prop} {Proposition} [section]
\newtheorem{thm}[prop] {Theorem} 
\newtheorem{defi}[prop] {Definition}
\newtheorem{lem}[prop] {Lemma}
\newtheorem{cor}[prop]{Corollary}
\newtheorem{prop-def}[prop]{Proposition-Definition}

\newtheorem*{thmA}{Theorem A} 
 
\newtheorem*{thmB}{Theorem B} 
\newtheorem*{thmC}{Theorem C}

\newtheorem{exam}[prop]{Example}
\newtheorem{rmk}[prop]{Remark}

\theoremstyle{remark}

\title{On polarizations of affine Calabi--Yau varieties}
\date{\today}
\bibliography{infty}

\author{Yueqiao Wu}

\address{Dept of Mathematics\\
  Johns Hopkins University\\
  3400 N. Charles Street\\
  Baltimore, MD 21218\\
  USA}
\email{ywu347@jh.edu}

\begin{document}

\maketitle

\begin{abstract}
    Let $(X, \Omega)$ be an affine Calabi--Yau variety, by which we mean a klt affine variety $X$ with a non-vanishing algebraic volume form $\Omega$. In connection with complete Calabi--Yau metrics of Euclidean volume growth on $(X, \Omega)$, we study algebraic aspects of polarizations of $X$ given by a negative valuation. We show that such a valuation is finitely generated with a klt degeneration if and only if $X$ admits a dlt log Fano type compactification, generalizing a compactification considered by Tian--Yau. We further give a valuative criterion detecting K-semistability of the degeneration by a negative valuation.
\end{abstract}

\setcounter{tocdepth}{1}
\tableofcontents

%
%
%
%
\section*{Introduction}
%
%
\subsection*{Motivation}
Throughout the paper we work over the complex numbers. To begin with, let $(X=\Spec R, p, \Omega)$ be a smooth affine variety of complex dimension $n$, where $\Omega$ is a nonvanishing algebraic volume form. We will also assume that $X$
has no non-constant invertible functions, so that $\Omega$ is unique up to constants. A natural question, going back to~\cite{TY1, TY2}, is: when does there exist a complete Calabi--Yau metric of Euclidean volume growth, that is, a complete K\"ahler metric $\omega$ such that 
\[
\omega^n=i^{n^2}\Omega\wedge \bar\Omega,
\]
and the volume of balls has maximal growth: $\vol_\omega(p, r)\sim r^{2n}$ as $r\to \infty$. Since the first construction of such canonical metrics in~\cite{TY2}, there have been many recent developments in classifying such metrics.

One of the most natural places where such varieties arise is as rescaled Gromov--Hausdorff limits of polarized compact K\"ahler--Einstein manifolds. In this case, via metric rescalings, there are two natural metric tangent cones associated to $X$. More precisely, as $\la\to \infty$, $(X, p, \la\omega)$ converges to the (local) metric tangent cone $C_0(X)$ at $p$, and as $\la\to 0$, $(X, p, \la\omega)$ converges to the metric tangent cone $C_\infty(X)$ at infinity, in the sense of (pointed) Gromov--Hausdorff convergence. In~\cite{DS2}, it was shown that both metric tangent cones are log Fano cone singularities with a Calabi--Yau cone metric, and that there is a 2-step algebraic degeneration from $X$ to each tangent cone. In general, such 2-step degeneration theory is subsequently obtained in~\cite{GLiu, Szeuniq, SZ26}. Thanks to the work of~\cite{CH22, SunZhang}, all such complete Calabi--Yau manifolds with quadratic curvature decay are classified by possible Calabi--Yau cones at infinity. This raises a purely algebro-geometric question: can we describe all possible algebraic degenerations from $X$ to K-semi/polystable log Fano cones? 

From the algebro-geometric perspective, the degeneration to the local tangent cone has been studied extensively over the past decade, commonly known as the stable degeneration conjecture. As it turns out, neither the local tangent cone nor the degeneration process depends on the metric; rather they only depend on the singularity. Given any klt singularity $(X, p)$, by minimizing a normalized volume function over all valuations centered at $p$, as first introduced in~\cite{Linvol}, the unique minimizing valuation gives the first step degeneration of $X$ to a K-semistable log Fano cone singularity, which then uniquely degenerates to a K-polystable log Fano cone singularity, thanks to the work of~\cite{BlumMin, Xuqm, LWX, XZstab}.

In contrast, it is known that the metric tangent cone at infinity depends on the metric. Even on $\C^n, n\ge 3$, it was understood that there are infinitely many different complete Calabi--Yau metrics of Euclidean volume growth, giving rise to infinitely many Calabi--Yau cones at infinity, see~\cite{GaborCY, CR}. Nevertheless, the degeneration process is still algebraic and is in fact given by a valuation $v_\omega$ induced by the Calabi--Yau metric on $X$. Such valuations do not admit a center on $X$, and only take negative values on non-constant global functions. The 2-step degeneration process asserts that $X$ degenerates, via $v_\omega$ to a K-semistable log Fano cone, which further degenerates to a K-polystable log Fano cone that serves as the metric tangent cone at infinity. In this paper, we describe some first attempts in studying these valuations. 

\subsection*{Compactifications of affine Calabi--Yau varieties}
From now on, let $(X=\Spec R, \Omega)$ be a klt affine Calabi--Yau variety, by which we mean a klt affine variety with no non-constant global invertible functions, together with a non-vanishing algebraic $n$-form $\Omega$. Any valuation $v$ on $X$ defines a filtration on $R$ by vector spaces, via:
\[
\cF^\la_vR\coloneqq \{f\in R\mid v(f)\ge -\la\}.
\]
Denote by $X^\val_\infty$ the space of valuations on $X$ which do not admit a center on $X$. We say a valuation $v\in X^\val_\infty$ is \emph{negative} if $v(f)<0$ for any $f\in R\setminus k$. The filtration defined by a bounded negative valuation $v$ satisfies $\cF^0R=k$, and that each $\cF^\la_vR$ is finite dimensional. It is also clear from the definition that negative valuations can only exist if $X$ admits no global invertible functions. We further say such a negative valuation is \emph{finitely generated} if the associated graded
\[
\gr_vR\=\bigoplus_{\la\in \R_+}\frac{\cF^\la_vR}{\cF^{<\la}_vR}
\]
is a finitely generated $k$-algebra, where $\Phi_v$ is the negative of non-positive values in the value group of $v$. All such finitely generated valuations are quasimonomial, and we get a natural degeneration from $X$ to a polarized affine cone 
\[
(X_v=\Spec \gr_vR, \xi_v)
\]
in the sense of~\cite{CSz18}. If $X_v$ further has klt singularities, then it is a log Fano cone singularity, whose K-stability theory has been very well studied, see e.g.~\cite{LX18, LWX}. There is a YTD correspondence for log Fano cones, which states that a log Fano cone $(X, \xi)$ is K-polystable iff it admits a weak Calabi--Yau cone metric. In particular, any metric tangent cone at infinity is a K-polystable log Fano cone, induced by a finitely generated negative quasimonomial valuation on $X$. 

Our first result explores geometric restrictions on $X$ for having such valuations. The first observation is that a necessary condition for $X_v$ to have klt singularities is imposed by a \emph{log discrepancy} function $A_X\colon X^\na\to \R\cup\{\pm\infty\}$. For a divisorial valuation $\ord_E$, this is defined as $A_X(\ord_E)\coloneqq 1+\ord_E(\Omega)$, and in general, this is the largest lsc extension of $A_X\colon \Xdiv\to \R$. 
In analogy with the usual klt condition, if $X_v$ has klt singularities then $A_X(v)<0$. With this, we state
\begin{thmA}
    Let $(X=\Spec R, \Omega)$ be a klt affine Calabi--Yau variety with $v_0$ a negative valuation. Then the following are equivalent:
    \begin{itemize}
        \item [(i)] $v_0$ is finitely generated with $X_{v_0}$ having klt singularities;
        \item [(ii)] $X$ admits a \emph{Koll\'ar compactification}, that is, a compactification $(\bar X, E=\sum_{i=1}^r E_i)$ of $X$ such that $E=\bar X\setminus X$, each $E_i$ is a negative divisorial valuation with $A_X(E_i)<0$, $(\bar X, E)$ is dlt, $-(K_{\bar X}+E)$ is ample and $v_0\in \QM(\bar X, E)$ is of rational rank $r$.
    \end{itemize}
\end{thmA}

As an immediate consequence of the above theorem and the 2-step degeneration theory, if $X$ is smooth and has a complete Calabi--Yau metric of Euclidean volume growth, then $X$ admits a Koll\'ar compactification. If $r=1$, that is, if $v_0$ is divisorial, then this gives back the geometric setup considered in~\cite{TY2}, where they showed, under the assumption that $E$ is smooth and K\"ahler--Einstein, that there exists a complete Calabi--Yau metric asymptotic to $X_{v_0}$. In general, this is a higher rank version of what is considered by Tian--Yau, and we refer to~\cite{CLCY} for similar generalizations in the collapsing case.

It is worth noting in contrast to the Tian--Yau setting, there is no assumption on K-stability of the log Fano cone $X_{v_0}$. In fact, the above compactification follows from a more precise characterization of the finite generation property of negative valuations, see Theorem~\ref{thm:kltfg}, analogous to the ones described in the Fano case and the local case, see~\cite{LXZ, XZstab, XuKstab}. 

More precisely, if $v_0=\ord_E$ is divisorial, then the Rees construction gives a Tian--Yau type compactification $(\bar X, E)$ such that $E=\bar X\setminus X$ and $K_{\bar X}+(1-A_X(E))E\sim_\Q0$. In this case, we have the identification $\cF^m_{v_0}R\simeq \Hnot (\bar X, mE)$ for $m$ divisible enough. In general, we generalize the notion of special $\Q$-complements introduced in~\cite{LXZ} and realize $v_0$ as an lc place of a special $\Q$-complement. A Koll\'ar compactification is obtained by running an MMP similar to the one in~\cite{XZstab}, by treating divisors at infinity as exceptional divisors. In particular, we do not run a relative MMP, and the Calabi--Yau condition plays an essential role in this process.

\subsection*{K-semistability of negative valuations}
Our next goal is to study K-stability of these negative valuations. If $v$ is finitely generated, then it is natural to define $v$ to be K-semistable if $(X_v, \xi_v)$ is a K-semistable log Fano cone. As is now well-understood, it is often useful to obtain a valuative criterion to characterize K-stability. 
Using the filtration induced by a negative valuation, we can also define the \emph{volume} of a valuation:
\[
\vol(v)\coloneqq \limsup_{\la\to \infty}\frac{\dim_k\cF^\la_vR}{\la^n/n!}\in (0, \infty].
\]
While this volume can be infinite, it is shown in~\cite{JW26} that it is finite for finitely generated valuations, and more generally for bounded negative valuations. We may also analogously define the \emph{relative volume} if we are given a second norm $\chi$ on $R$ that is linearly bounded. For each $\cF^\la_vR$, we may pick a common orthogonal basis $\{f_i\}$ of $\cF^\la_vR$ with respect to both $v$ and $\chi$, and define 
\[\vol_\la(v; \chi)\coloneqq \frac{1}{\la \dim\cF^\la_vR}\sum_i\chi(f_i).\]
The relative volume is given by
\[
S(v; \chi)=A_X(v)\lim_{\la\to \infty}\frac{\vol_\la(v; \chi)}{\vol_\la(v; v)}.
\]
We can now state the valuative criterion for K-semistability at infinity:
\begin{thmB}
    Let $v_0$ be a finitely generated negative valuation on $X$ with $X_{v_0}$ having klt singularities. Then $(X_{v_0}, \xi_{v_0})$ is K-semistable iff $A_X(v_0)<0$, and
    \[
    A_X(v)\ge S(v_0; v)
    \]
    for all $v\in X^\val_\infty$.
\end{thmB}
While this looks exactly like the valuative criterion for Fano varieties and log Fano cones, we mention 
that we have to test all valuations at infinity, whereas in the local case, it suffices to test all valuations centered at the singularity. The proof of the divisorial case is obtained by an inversion of adjunction argument using the techniques developed in~\cite{AZ22}; in general, the valuative criterion is obtained by a careful approximation by divisorial valuations.

Finally, while we cannot expect there to be a unique K-semistable valuation at infinity in general, we may still ask if all valuations satisfying the valuative criterion above are finitely generated, and thus give degenerations of $X$ to K-semistable log Fano cones. Somewhat surprisingly, this is not even clear when the valuation is divisorial. 

However, if there is a compactification close enough to the K-semistable valuation $v_0$, then we can perturb the compactification to the K-semistable one. For this, we need to measure how close two valuations are quantitatively. This follows the general philosophy of metric geometry on the space of norms studied in~\cite{JW26}, with its global analog studied in detail in~\cite{nakstab1}. In our setting, the following quantity is close to measuring the $\de_\infty$ distance between two negative bounded norms $\chi, \chi'$:
\[
T(\chi; \chi')\coloneqq \sup_{f\in R\setminus k}\frac{\chi'(f)}{\chi(f)}.
\]
With this, we can state:
\begin{thmC}
    Let $v_0$ be a K-semistable negative quasimonomial valuation, that is, it satisfies the valuative criterion, and suppose there is a log Fano type compactification $(\bar X, F)$ such that
    \[
    T(v_0; F)T(F; v_0)\le 1-A_X(F)^{-1}.
    \]
    Then $v_0$ is finitely generated, and $X_{v_0}$ is K-semistable.
\end{thmC}

This follows the general strategy adopted in local K-stability theory, that is, a K-semistable valuation will compute the log canonical threshold of a filtration defined by itself. In our case, a K-semistable valuation will compute $\lct(\chi_{v_0})\=\sup_{v\in X^\val_\infty}\frac{-A_X(v)}{T(v_0; v)}$, which serves the role of certain log canonical threshold at infinity. One may then realize $v_0$ as an lc place of a (sub)pair on $\bar X$, and the quantitative condition we give here ensures that we get a pair.

This is reminiscent of Tian--Yau type constructions, where the Calabi ansatz produces an almost Calabi--Yau reference metric, and one can then perturb it to be a genuine Calabi--Yau metric. In fact, such examples arise frequently when the metric tangent cone at infinity is irregular; in particular, the example studied in~\cite{CH15} satisfies the above quantitative condition.

\subsection*{Organization}The paper is organized as follows. We start with some preliminaries and recalling the theory of negative norms developed in~\cite{JW26} in~\S\ref{sec:prelim}, where we also define the log discrepancy and log canonical thresholds at infinity. In~\S\ref{sec:comp} we prove Theorem~A. In~\S\ref{sec:valinv} we define the relevant valuative invariants of valuations. Finally we define the valuative criterion and prove Theorem~B and Theorem~C in~\S\ref{sec:Kss}.

\subsection*{Acknowledgment}
This project started off as a joint project with Mattias Jonsson, and I thank him for many fruitful discussions and collaboration on~\cite{JW26}. I am grateful to Song Sun, from whom I first learned about the problem; and Ziquan Zhuang for explaining and answering many questions regarding~\cite{XZstab}. I also thank Fanjun Meng, Lingyao Xie, Chenyang Xu and Junsheng Zhang for helpful discussions. This work was partially supported by NSF grant DMS-1926686 and the Simons Collaboration Grant \emph{Moduli of varieties}. 

%
%
%
%
\section{Preliminaries}\label{sec:prelim}

We begin by recalling some results from~\cite{JW26} on bounded negative norms on an affine variety and their volumes. Besides setting up notation and terminologies, the main purpose of this section is to record a continuity result in local K-stability, and define log discrepancy and log canonical thresholds at infinity.
%
%
\subsection{Norms}\label{sec:norm}
Let $X=\Spec R$ be an affine variety. 
Using additive conventions, we define a (non-Archimedean) \emph{seminorm} on $R$ (or $X$) to be a function $\chi\colon R\to\R\cup \{\infty\}$ satisfying, for any $f, g\in \R$,
\begin{itemize}
    \item $\chi(f+g)\ge\min\{\chi(f),\chi(g)\}$;
    \item $\chi(fg)\ge\chi(f)+\chi(g)$;
    \item  $\chi(0)=\infty$ and $\chi|_{k^\times}\equiv0$.
\end{itemize}
We say $\chi$ is a \emph{norm} if in addition $\chi(f)<\infty$ for $f\ne 0$. 
Such norms are in one-to-one correspondence with (increasing, right-continuous, multiplicative and exhaustive) \emph{filtrations} on $R$ by vector spaces, via:
\[
\cF^\la_\chi R\coloneqq \{f\in R\mid \chi(f)\ge -\la\}, \quad \chi(f)\coloneqq -\inf\{\la\ge 0\mid f\in \cF^\la R\}.
\]

A seminorm is \emph{homogeneous} if $\chi(f^n)=n\chi(f)$ for all $f\in R$ and $n\ge1$. A (semi)norm is called a \emph{(semi)valuation} if $\chi(fg)=\chi(f)+\chi(g)$ for all $f,g\in R$. A norm $\chi$ is \emph{negative} if $\chi(f)<0$ for all $f\in R\setminus k$.

If $\chi,\chi'$ are seminorms, so is $\chi\wedge\chi'$ defined by 
\[(\chi\wedge\chi')(f)=\min\{\chi(f),\chi'(f)\}, f\in R.\]
Similarly $t\chi$ is a seminorm for any $t\in\R_+^\times$. These operations preserve the set of homogeneous seminorms, norms and negative norms, respectively.

\subsection{Geometric and bounded norms}
A large class of negative norms on $X$ is given by compactifications of $X$, which we now describe. Let $\bar X$ be a \emph{compactification} of $X$, that is, a projective variety having $X$ as an open affine subset, together with a \emph{strictly effective} Cartier divisor $D$ on $\bar X$, that is, an effective Cartier divisor $D$ with support equal to $\partial \bar X\=(\bar X\setminus X)_{\redu}$.
If $f\in R$ and $m\ge0$, then $fs_D^m$ is a rational section of $\cO_{\bar{X}}(mD)$, and a regular section for $m\gg1$. Here $s_D$ is the canonical section of $\cO_{\bar{X}}(D)$. We can therefore define a negative norm $\chi_{\bar{X},D}$ on $R$ by
\[
\chi_{\bar{X},D}(f):=-\min\{m\in\Z_{\ge0}\mid fs_D^m\in\Hnot(\bar{X},mD)\}.
\]
We call such a norm a \emph{geometric norm} if $D$ is further semiample. By~\cite[Lemma~6.7]{JW26}, if $\chi_0$ is a geometric norm and $\chi$ is any seminorm, then $\chi\ge C\chi_0$ for some $C>0$.

Let $\chi_0$ be a geometric norm on $R$. We say that a norm $\chi$ on $R$ is \emph{bounded} if there exists $C>0$ such that $C\chi_0\le\chi\le-C\chi_0$. We also say that $\chi$ is a \emph{bounded negative norm} if there exists $C\ge1$ such that $C\chi_0\le\chi\le C^{-1}\chi_0$. These notions do not depend on the choice of $\chi_0$, and that any geometric norm is a bounded negative norm.

\begin{defi}\label{defi:T}
    Given a bounded negative norm $\chi_0$ and any seminorm $\chi$ on $R$, we define
    \[
    T(\chi_0; \chi)=\sup_{f\in R\setminus k}\frac{\chi(f)}{\chi_0(f)}=\inf\{C\ge0\mid C\chi_0\leq \chi\}\in[0,\infty).
    \]
\end{defi}
It is not hard to see that $T(\chi_0;\chi)>0$ as soon as $\chi$ is \emph{centered at infinity} in the sense that $\chi(f)<0$ for some $f$.
If $\chi,\chi'$ are bounded negative norms, we set 
\[
  \de_\infty(\chi, \chi')\coloneqq \log \inf\{C\geq 1\mid C\chi\leq \chi'\leq C^{-1}\chi\}=\log\max\{T(\chi;\chi'),T(\chi';\chi)\}.
\]
One easily checks that $\de_\infty$ defines a complete metric on the space of bounded negative norms.
If $\chi$ and $\chi'$ are bounded negative norms, then so is $\chi\wedge\chi'$, as is $t\chi$ for any $t\in\R_+^\times$.
Extending the well-known result~\cite[Proposition~1.14]{BE21} or~\cite[Lemma~3.1]{AZ22} to the infinite dimensional setting, if $\chi_0$ is a bounded negative norm and $\chi$ is a bounded norm on $R$, then there exists a $k$-vector space basis $(f_i)_{i=0}^\infty$ of $R$, with $f_0=1$, that is orthogonal for both $\chi_0$ and $\chi$, see~\cite[Lemma~6.18]{JW26}.

If $h\colon R\to R'$ is a surjection of finitely generated $k$-algebras, then any seminorm $\chi$ on $R$ induces a \emph{residue seminorm} $\chi'$ on $R'$ via $\chi'(f'):=\sup_{f\in h^{-1}(f')}\chi(f)$. If $\chi$ is a norm (resp.\ bounded negative norm), so is $\chi'$.

For any bounded negative norm $\chi$, the \emph{volume} of $\chi$ is defined to be
\[
\vol(\chi)\coloneqq \limsup_{\la \to \infty}\frac{\dim\cF^\la_{\chi}R}{\la^n/n!}\in[0,\infty),
\]
where the limsup is in fact a limit.

%
%
\subsection{Finitely generated negative norms}
For a negative norm $\chi$, we denote by $\Phi_\chi\subset\R_+$ the subsemigroup of nonnegative values of the group generated by $-\chi(R\setminus 0)$. 
To any bounded negative norm $\chi$ we associate a \emph{Rees algebra}; this is the $\Phi_\chi$-graded $k[t^{\Phi_\chi}]$-algebra given by 
\[
\cR_\chi:=\bigoplus_{\la\in\Phi_\chi}t^\la\cF_\chi^\la R.
\]
The \emph{associated graded} of $\chi$ is the $\Phi_\chi$-graded $k$-algebra defined by 
\[
\gr_\chi R\coloneqq \bigoplus_{\la\in\Phi_\chi} \frac{\cF^\la_\chi R}{\cF^{<\la}_\chi R}.
\]
Note that $\chi$ is homogeneous (resp.\ a valuation) iff $\gr_\chi R$ is reduced (resp.\ an integral domain).
\begin{exam}
    Let $X=\A^n=\Spec k[z_1, \dots, z_n]$, and $v_\alpha$ a negative monomial valuation on $X$ with weight $\alpha=(\alpha_1, \dots, \alpha_n)\in \R^n_{<0}$. Concretely, 
    \[
    v_\alpha(f=\sum_{\beta\in\Z^n_+} a_\beta z^\beta)\= \min\{\langle \alpha, \beta\rangle\mid a_\beta\ne 0\}, 
    \]
    for any $f\in k[z_1, \dots, z_n]$. Then $\gr_{v_\alpha}k[z_1, \dots, z_n]=k[z_1, \dots, z_n].$ 
\end{exam}
Thanks to~\cite[Proposition~6.23]{JW26}, we define the class of finitely generated negative norms as follows.
\begin{defi}\label{defi:ft}
    A negative norm $\chi$ is \emph{finitely generated} if the following equivalent conditions are satisfied:
    \begin{itemize}
        \item[(i)]
        there exists a surjection $k[z_1,\dots,z_m]\to R$, such that $\chi$ is the residue norm induced by a negative monomial valuation on $k[z_1,\dots,z_m]$;
         \item[(ii)]
         $\gr_\chi R$ is a finitely generated $k$-algebra;
    \item[(iii)]
        $R_\chi$ is finitely generated $k[t^{\Phi_\chi}]$-algebra.
    \end{itemize}
\end{defi}
Any finitely generated norm is bounded. 
The Rees correspondence provides a one-to-one correspondence between the set of equivalence classes of $\Z$-valued finitely generated negative norms and the set of polarized compactifications $(\bar X, D)$, where $\bar X=\Proj\cR_\chi$ and $D$ is a strictly effective and ample $\Q$-Cartier $\Q$-divisor.


%
%
\subsection{Berkovich spaces and valuations}
The \emph{Berkovich analytification} $X^\na$ of $X$ is defined as the space of semivaluations on $X$.
We equip $X^\na$ with the topology of pointwise convergence on $R$; then $X^\na$ is locally compact. 
A semivaluation $v\in X^\na$ is \emph{centered on $X$} if $v\ge0$ on $R$: in this case the \emph{center} $c_X(v)$ is the point on $X$ defined by the prime ideal $\{v>0\}\subset R$. If $v$ is not centered on $X$, we say that $v$ is \emph{centered at infinity}. We write $X^\beth$ (resp.\ $X^\na_\infty$) for the closed (resp.\ open) set of semivaluations centered on $X$ (resp.\ at infinity). 

If $\bar X$ is a compactification of $X$, then by the valuative criterion of properness, any valuation $v\in X^\val$ admits a center $c_{\bar X}(v)\in\bar{X}$, given by the unique point such that $v$ is nonnegative on the local ring $\cO_{\bar X,c_{\bar X}(v)}$ and strictly positive on its maximal ideal. The center is defined also for $v\in X^\na$. We have $c_{\bar X}(v)\in X$ iff $v\in X^\beth$, and in this case $c_{\bar X}(v)=c_X(v)$ as defined above.
We have a natural multiplicative action of $\R_{>0}$ on $X^\na$, and that $X^\beth$ and $X^\na_\infty$ are invariant under this action. Moreover, $X^\na_\infty/\R_+^\times$ is compact; see~\cite[Proposition~2.1]{JW26}. 

A valuation $v\in X^\val$ is \emph{quasimonomial(qm)} if $\ratrk(v)+\trdeg(\kappa(v)/k)=n$, where $\ratrk(v)=\dim_\Q (v(k(X))^\times\otimes_\Z\Q)$ and $\kappa(v)=\{v\ge0\}/\{v>0\}$ is the residue field of $v$. We say $v$ is \emph{divisorial} if in addition $\ratrk(v)=1$.
We write $X^\div$ (resp. $\Xqm$) for the set of divisorial (resp. qm) valuations, and $X^\div_\infty$ (resp. $\Xqm_\infty$) for the subset of divisorial valuations centered at infinity. These sets are invariant under the scaling action by $\Q_{>0}$.

Valuations in $\Xqm_\infty$ can be described geometrically. Let $(Y, E=\sum_{i=1}^rE_i)$ be a log smooth projective model of $X$, by which we mean a log smooth pair together with a projective birational morphism $\pi: Y\to \bar X$ to some compactification of $X$. Suppose $Z\subseteq \bigcap_{i=1}^r E_i$ is an irreducible component and $\alpha\in \R^r_+$.
Then if we let $\eta$ be the generic point of $Z$, and pick local algebraic coordinates around $\eta$ so that $E_i=(y_i=0)$, we may write any $f\in \cO_{Y, \eta}$, as $f=\sum_\beta c_\beta y^\beta$ with $c_\beta\in \widehat{\cO}_{Y, \eta}$ either zero or a unit, and set
\[
v_\alpha(f)=\min\{\langle \alpha, \beta\rangle \mid c_\beta\ne 0\}.
\]
It follows that $v_\alpha$ is a qm valuation, and $v_\alpha\in \Xqm_\infty$ if $\pi(c_Y(v))\subset\partial\bar X$. In fact, let us remark that the above construction works whenever $(Y,E)$ is an snc pair, and we denote by $\QM(Y,E)$ the set of qm valuations described by the model $(Y,E)$.
By~\cite[Proposition~3.7]{JM12}, given any $v\in\Xqm_\infty$, there exists a log smooth model $(Y, E)$ with $v\in \QM(Y,E)$.
In the special case when $v\in X^\div_\infty$, there exists an integrally closed compactification $\bar{X}$ of $X$, an irreducible component $E\subset\partial\bar X$, and $c\in\R_{>0}$ such that $v=c\ord_E$ by~\cite[Lemma~2.45]{KollarMori}. 

Following the terminologies in \S\ref{sec:norm}, we say $v\in \Xval$ is \emph{negative} if $v(f)<0$ for all $f\in R\setminus k$. A negative valuation is centered at infinity, but not every valuation centered at infinity is negative. In fact, we have the following simple observation:
\begin{lem}
    If $X$ admits a negative valuation, then there are no non-constant global invertible functions on $X$, i.e., units of $R$ are $k^\times$.
\end{lem}
Not all negative quaismonomial valuations are bounded. However, we recall the following corollary.
\begin{cor}[{\cite[Corollary~6.27]{JW26}}]
    Any finitely generated negative valuation is quasimonomial.
\end{cor}

We also have the following special case of~\cite[Lemma~6.31]{JW26}:
\begin{lem}\label{lem:fgres}
    Let $(\bar X, D)$ be a compactification of $X$ with $D$ strictly effective and ample, and $v$ a negative valuation on $X$. If $v$ is finitely generated on $\bar X$, it is also finitely generated on $X$.
\end{lem}

%
%
\subsection{Fubini--Study operators}
Analogous to the function induced by a graded sequence of ideals, any bounded negative norm $\chi$ induces a lsc function $\FS(\chi): X^\na\to [0,\infty)$ via 
\[\FS(\chi)(v)\coloneqq T(\chi;v).\]
It is further continuous if $\chi$ is finitely generated, see~\cite[Theorem~C]{JW26}.
If $\chi=\chi_{\bar X, D}$ is geometric, then 
\begin{equation}\label{equ:FSgeom}
    \FS(\chi)(v)=v(D).
\end{equation}
More generally, let $\chi'=\chi_{\bar X, D}$ be a fixed geometric norm, and $\chi$ a bounded negative norm with $\chi\le C^{-1}\chi'$ for some $C>0$. Then this gives an identification of $\cF^m_\chi R\subseteq H^0(\bar X, \lfloor Cm\rfloor D)$ as a graded linear series. Let $\fb_\bullet(\chi) = (\fb_m(\chi))_m\=(\fb|\cF^m_\chi R|)_m$ be the associated base ideals. Then the same proof as in~\cite[Lemma~8.8]{JW26} gives
\begin{equation}\label{equ:FSbaseideal}
\FS(\chi)(v)= v(CD)-v(\fb_\bullet(\chi)).
\end{equation}

\begin{lem}\label{lem:FScts}
    If $v_i$ is a sequence of bounded negative valuations converging to $v_0$ with respect to the metric $\de_\infty$, and that $v_0$ is finitely generated, then $\FS(\chi_{v_i})(v)\to \FS(\chi_{v_0})(v) $ locally uniformly.
\end{lem}

\begin{proof}
    Since $\FS(\chi_{v_0})$ is continuous,
    by homogeneity, it suffices to check convergence on the compact subset $K=\{\FS(\chi_{v_0})(v)=1\}\subset X^\an_\infty$. If $\delta_i\=\de_\infty(v_0; v_i)$, then
    \[
    e^{-\delta_i}=e^{-\delta_i}T(v_0; v)\le T(v_i; v)\le e^{\delta_i}T(v_0; v)=e^{\delta_i}
    \]
    Thus $T(v_i; v)\to T(v_0; v)$ uniformly on $K$
    as $i\to\infty$.
\end{proof}

\subsection{Local K-stability}\label{sec:locKstab}
In this section, we review the notion of a log Fano cone singularity, and record a continuity result of the local $\delta$ invariant with respect to Reeb vector fields.
\begin{defi}
    Let $X=\Spec R$ be a normal affine variety and let $\T\cong \G_m^r$ be an algebraic torus acting on $X$. The action is \emph{good} if it is effective and $x\in X$ is in the closure of any torus orbit. Such a triple $(X, \T, x)$ is a $\T$-variety.
\end{defi}
The torus action on $X$ yields a natural weight decomposition of $R$. Let $N=\Hom(\G_m, \T)$, and put $M=N^\vee$ to be the weight lattice. Then
\[
R=\bigoplus_{\alpha\in M }R_\alpha,
\]
with $R_0=k$. The \emph{Reeb cone} is defined to be
\[\ft_\R^+\coloneqq \{\xi\in N_\R\mid \langle \xi, \alpha\rangle >0 \text{ for all } \alpha \text{ such that } R_\alpha\ne 0\}.
\]
A \emph{Reeb field } $\xi\in \ft_\R^+$ defines a $\T$-invariant quasimonomial valuation $v_\xi$ on $X$, via
\[
v_\xi(f_\alpha)=\langle \xi, \alpha\rangle, \text{ for any } f_\alpha\in R_\alpha, \quad v(\sum f_\alpha) = \min\{\langle \xi, \alpha\rangle \mid f_\alpha\ne 0\}.
\]
\begin{defi}
    A \emph{log Fano cone singularity} is $x\in (X, \Delta, \T, \xi)$, where $x\in (X,\Delta)$ is a klt singularity with a good $\T$-action, and $\xi\in \ft_\R^+$.
\end{defi}
When $\xi$ is rational, $\xi$ generates a $\G_m\cong \langle \xi\rangle$ action on $X$, making $X\setminus\{x\}$ a Seifert $\G_m$-bundle over a log Fano variety $(E, \Delta_E)$. In general, $\xi$ generates a torus action with rank being the rational rank of $v$. 

Analogous to the log Fano case, K-stability of log Fano cones can be detected using valuative invariants and basis type divisors which we now recall. 
For a fixed $m\in \Z_{>0}$, an $(m, \xi)$-basis type divisor is a $\T$-invariant $\Q$-divisor on $X$ of the form
\[
D_m(\xi)=\frac{(n+1)A(\xi)}{nmN_m}\sum_{i=1}^{N_m}\{f_i=0\},
\]
where $N_m=\dim R_m\coloneq \dim \bigoplus_{\langle\xi, \alpha\rangle\le m}R_\alpha$, and $\{f_i\}$ is a $\T$-invariant basis of $R_m$. For each $\T$-invariant $v\in \Val_{X, x}$, set 
\[
S_m(\xi; v)\= \sup\{v(D_m(\xi))\mid D_m(\xi) \text{ is an } (m, \xi)\text{-basis type divisor}\},
\]
where supremum is achieved when $\{f_i\}$ are chosen to be compatible with $v$.
We also define
\[
\delta_m(X, \Delta; \xi) = \inf\{\lct_x(X, \Delta; D_m)\mid D_m \text{ is an $(m, \xi)$-basis type divisor} \} = \inf_{v}\frac{A_{X, \Delta}(v)}{S_m(\xi; v)},
\]
where the latter infimum runs over all $v\in \Val^{\T, *}_{X, x}$, the set of $\T$-invariant valuations centered at $x$ with finite log discrepancy. Using an Okounkov body argument as in~\cite{NALFC1, Huang} (cf.~\cite{BlJ}), we have that the limit
\[
S(\xi; \cF)\=\lim_{m\to \infty} S_m(\xi; \cF)
\]
exists, and for any $\varepsilon>0$, there is $m(\varepsilon)>0$ such that for any $m\ge m(\varepsilon)$, one has
\begin{equation}\label{equ:Sapprox}
S_m(\xi; \cF)\le (1+\varepsilon)S(\xi; \cF)
\end{equation}
uniformly for any $\cF$. In particular, this allows for the following definition and characterization of the local $\d$-invariant:
\[
\delta(X, \Delta; \xi)\= \inf_{v\in \Val_{X, x}^{\T, *}}\frac{A_X(v)}{S(\xi; v)} = \lim_{m\to \infty} \delta_m(X, \Delta; \xi).
\]
As noted before, when $\xi$ is rational, $X$ is an orbifold cone over a log Fano pair $(E, \Delta_E)$. In this case, as observed in~\cite{LW}, we have
\[
\delta(X, \Delta; \xi)=\min\{1, \delta(E, \Delta_E)\}.
\]
The local $\delta$-invariant is in fact continuous with respect to the Reeb field. 
\begin{prop}\label{lem:deltaapprox}
    Let $(X, \Delta; \xi)$ be a K-semistable log Fano cone singularity with $\xi_i\to \xi$. Then 
    \[
    \lim_i\delta(X, \Delta; \xi_i)=\delta(X, \Delta; \xi).
    \]
\end{prop}
\begin{proof}
    First note that if $\xi$ is rational, then $\delta(X, \Delta; \xi) =\min\{1, \delta(V, \Delta_V)\}$ where $(V, \Delta_V)$ is the log Fano quotient of $X$ by $\xi$. In particular, $\delta(X, \Delta; \xi)\le 1$ unless $\xi$ is the unique stable rational Reeb field on $X$, in which case the torus action on $X$ has rank $1$, and the conclusion automatically holds. We will thus assume in what follows that the torus action has a higher rank. We may also assume $\xi_i\ge \xi$ since $\delta$-invariants are scaling invariant. Set
    \[
    \delta_m(X, \Delta; \xi)\coloneqq \inf\{\lct_x(X, \Delta; D)\mid D \text{ is a $\T$-invariant $(m, \xi)$-basis type divisor}\},
    \]
    Then $\delta(X, \Delta; \xi) = \lim_m \delta_m(X, \Delta; \xi)$. To simplify notation, we will write $\delta_i\coloneqq \delta(X, \Delta; \xi_i)=\lim_m\delta_{i,m}\coloneqq \lim_m\delta(X, \Delta; \xi_i)$ and $\delta\coloneqq \delta(X, \Delta; \xi)=\lim_m \delta_m\coloneqq \lim_m \delta_m(X, \Delta; \xi)$.

    Since $S(\xi_i; v)\to S(\xi; v)$ for any valuation $v$ centered at $x$, by upper-semicontinuity, we always have $\delta \ge \limsup_i \delta_i$. It is also clear that $\liminf_i\delta_i\ge\delta$ by our choice of $\xi_i$.
    Suppose now that this inequality is strict, i.e. there is some $c\in \R$ such that $\limsup_i\delta_i<c<\delta$. For fixed $c'\in (c, \delta)$ and $0<\varepsilon\ll1$, we may pick $i\gg 0$ such that 
    \[
    c'\frac{A(\xi)\vol(\xi_i)}{A(\xi_i)\vol(\xi)}(1-\varepsilon)>c.
    \]
    Note also that if $\delta_i<c$, then $\delta_{i,m}<c$ for $m=m(i)\gg 0$, and in particular, for each fixed $i$, after potentially passing to larger $m$, we may choose $m(i)$ such that whenever $m\ge m(i)$, $\frac{A(\xi)N_{m,i}}{A(\xi_i)N_m}>(1-\varepsilon)\frac{A(\xi)\vol(\xi_i)}{A(\xi_i)\vol(\xi)}$, where $N_m = \dim \bigoplus_{\langle \xi, \alpha\rangle\le m}R_\alpha$, and $N_{m,i}=\dim \bigoplus_{\langle \xi_i, \alpha\rangle\le m}R_\alpha$. Now for any fixed $m\ge m(i)$, we can find $(m, \xi_i)$-basis type divisor $D$ with $(Y, cD)$ is not klt at $x$. Since $\xi_i\ge \xi$, we can make $D$ into an $(m, \xi)$-basis type divisor by possibly adding additional basis:
    \[
    D(\xi)\coloneqq \frac{(n+1)A(\xi)}{nmN_m}\left(\sum \{f_i=0\}+\sum_j\{g_j=0\}\right)=\frac{A(\xi)N_{m,i}}{A(\xi_i)N_m}D+D',
    \]
    where $(g_j)$ form a basis of $\bigoplus_{\langle \xi, \alpha\rangle \le m<\langle \xi_i, \alpha\rangle}R_\alpha$. Then for the above choice of $i$ and $m$, we have
    \[
    c'D(\xi)= c'\frac{A(\xi)N_{m,i}}{A(\xi_i)N_m} D+D'\ge cD+c'D'.
    \]
    This, combined with $(X, \Delta; cD)$ being not klt at $x$, implies that $(X, \Delta; c'D)$ is not klt at $x$ either, i.e., $\delta_m<c'$. Since this is true for any $m\gg 0$, this gives a contradiction as $\delta=\lim_m\delta_m> c'$.
\end{proof}

We record the following comparison of log canonical thresholds of basis type divisors for later use.
\begin{lem}\label{lem:lctadj}
    Let $x\in(X, \Delta; \xi)$ be a log Fano cone with $\xi$ rational and let $(E, \Delta_E)$ be the log Fano base. Suppose $D$ is a divisor on $E$ such that the corresponding divisor $D_m$ on $X$ obtained by pulling back is an $(m, \xi)$-basis type divisor. Then 
    \[
    \lct_x(X, \Delta; D_m)\ge \min\{\frac{A_{X,\Delta}(\xi)}{S_m(\xi; v_\xi)}, \lct(E, \Delta_E; D)\},
    \]
    with $\frac{A_{X,\Delta}(\xi)}{S_m(\xi; v_\xi)}\to 1$ as $m\to \infty$.
\end{lem}
\begin{proof}
    Let $\pi: Y\to X$ be the plt blowup at $x$ extracting $F\cong E$. Then $X$ is the affine cone over $E$ with polarization $L\sim_\Q -A_{X,\Delta}(F)^{-1}(K_E+\Delta_E)$. Write $\tilde D$ for the strict transform of $D_m$. 
    By inversion of adjunction, if $(F, \Delta_F+ t \tilde D|_F)\cong (E, \Delta_E+tD)$ is lc, then $(Y, \pi^{-1}_*\Delta+F+t\tilde D)$ is also lc near $F$. Note that
    \[
    K_Y+\pi^{-1}_*\Delta+F+t\tilde D=\pi^*(K_X+\Delta)+A_{X, \Delta}(F)F+t\tilde D \ge \pi^*(K_X+\Delta+ t'D),
    \] 
    where $t'=\min\{\frac{A_{X,\Delta}(F)}{\ord_F(D_m)}, t\}$, we have $(X,\Delta+t'D)$ is also lc at $x$. This shows
    \[
    \lct_x(X,\Delta; D_m)\ge \min\{\frac{A_{X,\Delta}(F)}{\ord_F(D_m)}, \lct(E, \Delta_E; D)\}.
    \]
    Moreover, note that $\ord_F$ is nothing but the valuation given by a rational multiple Reeb field $\xi$, a direct computation shows
    \[
    \frac{A_{X,\Delta}(\xi)}{S_m(\xi; v_\xi)}=\frac{A_{X,\Delta}(F)}{\ord_F(D_m)} = \left(\frac{n+1}{nmN_m}\sum_{d=1}^m dh^0(E, dL)\right)^{-1}\to 1.
    \]
\end{proof}

%
%
\subsection{Log discrepancy and log canonical thresholds at infinity}
Throughout this section, $X$ is a klt affine variety with no non-constant global invertible functions, and Calabi--Yau in the sense that we are given an (algebraic) nonvanishing $m$-th pluricanonical section $\tilde\Omega$ on $X$, and put $\Omega\=\tilde\Omega^{\otimes 1/m}$. In this case, such volume form is unique up to scaling.
Our goal is to define the log discrepancy function on $X^\na$. The log discrepancy function will also play a crucial role later in the article in analyzing singularities at infinity. 

Before going into precise constructions, let us first remark that all of the constructions in this section work without the assumption that $X$ has no global invertible functions. When there are non-constant global invertible functions, $\Omega$ is then non-unique, and the constructions below will depend on the choice of $\Omega$.

Generalizing a log discrepancy function first defined in~\cite{JonDynBerko}, we define a \emph{log discrepancy function} $A_{X}\colon X^\div\to\Q$ by setting 
\[
A_{X}(t\ord_E)=t(1+\ord_E(\Om))
\]
for $t\in\Q_{\ge0}$ and $E$ a prime divisor over $X$, and we write $A_X$ when $\Delta=0$. More precisely, if $\bar{X}$ is a normal compactification of $X$, then there exists a projective birational morphism $\pi\colon Y\to \bar X$, with $Y$ normal, and a prime divisor $E\subset Y$, such that $v=t\ord_E$. Then $\Omega$ extends as a rational form on $Y$, and $\ord_E(\Omega)$ is its order of vanishing along $E$.
Its restriction to the set of divisorial valuations centered on $X$ agrees with the usual log discrepancy function on $\Q$-Gorenstein varieties. 
\begin{lem}\label{lem:ldpair}
    There exists a smooth projective variety $Y$, a birational map $Y\to \bar X$ for some compactification $\bar X$ of X, and a $\Q$-Cartier divisor $D$ on $Y$ with snc support such that the log discrepancy function $A_{X}$ on $X^\div=Y^\div$ is the log discrepancy function of the pair $(Y,D)$ in the sense that 
    \begin{equation}\label{equ:ldpair}
       A_{X}(v)=A_Y(v)-v(D)
    \end{equation}
   for $v\in X^\div=Y^\div$. 
\end{lem}
\begin{proof}
    Let $\bar X$ be a compactification of $X$ and $\pi\colon Y\to \bar X$ a log resolution. Then the extension of $\Omega$ on $Y$ defines a $\Q$-Cartier divisor $D\=-\div_Y(\pi^*\Omega)$ on $Y$ with snc support. Then $A_X(v)=A_Y(v)-v(D)$ for all $v\in\Xdiv$, since for any prime divisor $E\subset Y'\xrightarrow[]{\mu} Y$, $\ord_E(\Omega)=\ord_E(K_{Y'/Y})+\ord_E(-D)$. 
\end{proof}
We emphasize that the divisor $D$ is not necessarily effective. More generally, for any normal compactification $\bar X$ of $X$, let $K_{\bar X}=\sum_E\ord_E(\Omega)E$, where $E$ runs through all irreducible components of $\partial\bar X$, be the Weil divisor defined by $\Omega$. Then
\[
 K_{\bar X}+(1-A_X(\ord_E))E=K_{\bar X}-\sum_E\ord_E(\Omega)E =0.
\]

\begin{cor}
    The log discrepancy function $A_X\colon X^\div\to\Q$ is lsc.
\end{cor}
\begin{proof}
    The function $v\mapsto v(D)$ on $Y^\div$ is continuous, whereas $v\mapsto A_Y(v)$ is lsc, see~\cite[Lemma~5.7]{JM12}.
\end{proof}
As $X^\div$ is dense in $X^\na$, we can now make the following definition.
\begin{defi}
    The log discrepancy function $A_X\colon X^\na\to\R\cup\{\pm\infty\}$ is the largest lsc extension of $A_X\colon X^\div\to\Q$. In other words,
    \[
    A_X(v)=\liminf_{X^\div\ni w\to v}A(w)
    \]
    for any $v\in X^\na$.
\end{defi}
This definition is not so practical. Following~\cite[Appendix A]{nakstab2}, we can give some more precise information.
\begin{lem}
    The log discrepancy $A_X\colon X^\na\to\R\cup\{\pm\infty\}$ has the following properties:
    \begin{itemize}
        \item[(i)]  
        $A_X(v)>-\infty$ for all $v\in X^\val$;
        \item[(ii)]  
        for $Y$ as in Lemma~\ref{lem:ldpair}, the restriction of $A$ to $X^\val=Y^\val$ satisfies~\eqref{equ:ldpair}, where $A_Y$ denotes the log discrepancy function in~\cite{JM12}.
    \end{itemize}
\end{lem}
\begin{proof}
    Denote by $A_X'\=A_{Y,D}$. Since $A_Y$ is lsc, $v\mapsto v(D)$ is continuous, we have that $A_X'$ is lsc, and hence $A_X'\le A_X$ by definition of $A_X$. On the other hand, for any log smooth model $(Y', D')$ over $Y$, $A_Y\circ r_{Y', D'}$ is continuous, where $r_{Y', D'}: Y^\val\to \QM(Y', D')$ is the retraction map. Hence $A_X'\circ r_{Y', D'}$ is also continuous, and one has $A_X'\circ r_{Y', D'}\ge A_X\circ r_{Y', D'}$, since $A_X$ is lsc, and they agree on the dense subset $\Xdiv$. Now $A_X'=\sup_{(Y', D')}A_X'\circ r_{Y', D'}\ge \liminf_{(Y', D')}A_X\circ r_{Y', D'}\ge A_X$. This finishes the proof of~(ii), from which~(i) follows, see~\cite[Proposition~A.4]{nakstab2}.
\end{proof}

For later use, let $\Delta=\frac 1k\{f=0\}\ge0$ be an effective $\Q$-Cartier $\Q$-divsior for some $f\in R$, we also put 
\[A_{X,\Delta}(v)\=A_X(v)-v(\Delta)\=A_X(v)-\frac 1kv(f),\]
for $v\in X^\val$.
Recall that if $\fa_\bullet$ is a graded sequence of $\fm_x$-primary ideals for $x\in X$ a closed point, then $\lct(\fa_\bullet)$ satisfies
\[
\lct(\fa_\bullet)=\inf_{v}\frac{A_X(v)}{v(\fa_\bullet)},
\]
where $v$ runs through all valuations of finite log discrepancy centered at $x$. The role of the function $v\mapsto v(\fa_\bullet)$ is played by the function $\FS(\chi)$ for a bounded negative norm $\chi$. Drawing the analogy, we define similarly the log canonical threshold of a bounded negative norm:
\begin{defi}
    The \emph{log canonical threshold} of a bounded negative norm is defined by 
    \[\lct(\chi)\coloneqq\sup_{v\in X^\val_\infty}\frac{-A_X(v)}{\FS(\chi)(v)}.\]
\end{defi}
\begin{rmk}
    We have $\lct(\chi)=\lct(\chi^\hom)$ by~\cite[Lemma~8.6]{JW26}.
\end{rmk}
\begin{lem}\label{lem:lctgeld}
    For any bounded negative valuation $v$ we have $\lct(\chi_v)\ge-A(v)$.
\end{lem}
\begin{proof}
    This follows from unwinding the definitions and using $\FS(\chi_v)(v)=1$.
\end{proof}
%
%
\subsection{Polarized affine Calabi-Yau varieties}\label{sec:PCY}
The notion of polarized affine Calabi--Yau varieties was first introduced in~\cite{Sunbubble}. 
\begin{defi}
    A \emph{polarized affine Calabi--Yau variety} is a pair $(X, \Omega, v_0)$ with $(X=\Spec R, \Omega)$ a klt affine Calabi--Yau variety and $v_0$ a finitely generated negative valuation 
    on $X$ such that $X_{v_0}=\Spec\gr_{v_0}R$ is klt.
\end{defi}
While detailed study of the geometry of polarized affine Calabi--Yau varieties will be postponed to later sections, we give some examples and remarks relating them to more familiar objects that already appeared in the literature. 
\begin{exam}
    For a finitely generated negative valuation $v_0$, $(X_{v_0}\coloneqq \Spec \gr_{v_0}R, \xi_{v_0})$ is a polarized affine cone in the sense of~\cite{CSz18}, where $\xi_{v_0}$ is the Reeb vector corresponding to the (positive) grading on $\gr_{v_0}\!R$. It is a log Fano cone singularity if it is further klt.
\end{exam}
In general, for an affine variety $X$ with a finitely generated negative valuation $v_0$, the infinity cone $X_{v_0}$ may not have klt singularities.
We will see later that $A_X(v_0)<0$ is a necessary condition that ensures this. For the moment, 
Let us illustrate this by comparing to other affine varieties admitting a complete Calabi--Yau metric, as studied in~\cite{TY1, CLCY}:
\begin{exam}
    Let $(\bar X, E)$ be a plt log Calabi--Yau pair, $X=\bar X\setminus E$, and $v_0=\ord_E$. Then by adjunction $X_{v_0}$ is the affine cone over a log Calabi--Yau pair $(E, \Delta_E)$, which is lc at the cone point, by~\cite[Lemma 3.1]{SingMMP}. Note that in this case $A_X(v_0)=0$.
\end{exam}
More generally, we have:
\begin{rmk}
    Suppose $(X, \Omega)$ is affine log Calabi--Yau in the sense of~\cite{KY}, that is, $X=\bar X\setminus D$ where $(\bar X, D)$ is a dlt log Calabi--Yau pair, and $-D$ is the (Weil) divisor defined by $\Omega$. Note that here $X$ may have non-constant invertible functions, and the choice of $\Omega$ is relevant when defining the log discrepancy function. With this choice, we have $A_\Omega=A_{\bar X, D}\ge0$. We refer the readers to~\cite{Liu26} for more details on the study of finite generation of valuations in this setting.
\end{rmk}
In view of this remark, together with the result of~\cite{CLCY}, the log discrepancy condition should morally be thought of as a complex structure non-collapsing condition.



\section{Compactifications of polarized affine Calabi--Yau varieties}\label{sec:comp}
We now study the geometry of polarized affine Calabi--Yau varieties in more detail. The main goal is to prove Theorem A. This will be done in two steps. The divisorial case follows from the Rees construction; in general, we exploit techniques used in higher rank finite generation theory, as developed in~\cite{LXZ, XZstab}.
%
%
\subsection{Models and $\Q$-complements}
We first set up and recall some terminologies that we will use later. Throughout, we fix $X$ to be a klt affine Calabi--Yau variety.
While a general affine variety $X$ may have many different compactifications, we are mostly interested in compactifications with special geometric structures. A prototypical example of such compactifications is given by the setting in~\cite{TY2}. This motivates us to make the following definition. 
\begin{defi}
    A \emph{Tian--Yau type compactification} of $X$ is a plt pair $(\bar X, E)$ such that $\bar X$ is a compactification of $X$, the prime divisor $E$ is the scheme-theoretic complement of $X$ in $\bar X$ with $A_X(E)<0$, and that $-(K_{\bar X}+E)$ is ample.
\end{defi}
We remark that not all affine varieties will admit a Tian--Yau type compactification, since this in particular implies that $\ord_E$ defines a negative divisorial valuation on $X$. 
On the other hand, by adjunction, the divisor $E$ appearing in a Tian--Yau type compactification has klt singularities, and if we write $(K_{\bar X}+E)|_E=K_E+\Delta_E$, then $(E, \Delta_E)$ is a log Fano pair. This suggests that $E$ should be thought of as an analog of a Koll\'ar component over a klt singularity, which we refer to~\cite{LXKcomp} for a detailed study.

Slightly generalizing the notion of compactifications, we define models of $X$ to allow resolutions of $X$.
\begin{defi}
    A \emph{log smooth (resp. toroidal) model} $(Y, E)$ of $X$ is a projective log smooth (resp. toroidal) model over a compactification of $X$, which restricts to an isomorphism over the smooth locus of $X$.
\end{defi}
\begin{rmk}
    If $X$ is smooth, a log smooth (resp. toroidal) model of $X$ is a compactification of $X$.
\end{rmk}

We also recall the definition of ($\Q$-)complements. For our purposes, we shall need a notion of special $\Q$-complements which is slightly more general than the ones defined in~\cite{LXZ}.
\begin{defi}
    Given a log canonical pair $(\bar X, \Delta)$ with $-(K_{\bar X}+\Delta)$ ample, a \emph{$\Q$-complement} is an effective $\Q$-Cartier $\Q$-divisor $\Gamma\sim_{\Q} -(K_{\bar X}+\Delta)$ such that $(\bar X, \Delta+\Gamma)$ is log canonical. A $\Q$-complement is an $N$-complement for some $N\in \Z_{>0}$ if one further has $N(K_{\bar X}+\Delta+\Gamma)\sim 0$.
\end{defi}

\begin{defi}
    Let $(\bar X, F)$ be a Tian--Yau type compactification of $X$.
    We say a $\Q$-complement $\Gamma$ of $(\bar X, F)$ is \emph{special with respect to a log smooth (resp. toroidal) model $(Y,E)$} if there is a rational map $p: Y\dashrightarrow \bar X$ which is an isomorphism over the smooth locus of $X$, such that the birational transform $p^{-1}_*\Gamma\ge G$ for some effective ample $\Q$-divisor $G$ on $Y$ whose support does not contain any stratum of $(Y,  E)$. We say that a valuation $v\in X^\val_\infty$ is a \emph{monomial lc place of $\Gamma$} if $v\in \lc(\Gamma; Y, E)\coloneqq \QM(Y, E)\cap  \mathrm{LC}(\bar X, F+\Gamma)$.
\end{defi}
%
%
\subsection{Tian--Yau type compactification}
In this section, we prove Theorem~A for divisorial valuations. As a preliminary step toward the general case, we also show that any polarized affine Calabi--Yau variety admits a Tian--Yau type compactification.


As noted previously, if $v_0$ is divisorial, then the proj construction gives a projective compactification of $X$. 
In fact, this compacfication is of Tian--Yau type. For later use, we record a slightly more general version of this observation. For a $\Q$-Cartier $\Q$-divisor $D=\frac 1m \{f=0\}$ on $X$ and a finitely generated negative valuation $v$, we will write $D_v\=\frac 1m \{f_v=0\}$, where $f_v$ is the image of $f$ in $\gr_vR$. 
\begin{lem}\label{lem:divfg}
    Let $v_0$ be a negative divisorial valuation on $X$, and $D=\frac1k\{f=0\}\ge0$ a $\Q$-Cartier $\Q$-divisor as above so that $(X, D)$ remains klt. 
    Then $v_0$ is finitely generated with $(X_{v_0}, D_{v_0})$ having klt singularities iff 
    there exists a compactification $(\bar X, \bar D+E)$ of $X$, where $\bar D$ is the closure of $D$ in $\bar X$,
    such that there is a unique strictly effective $\Q$-Cartier prime divisor $E=\bar X\setminus X$ with $v_0=c\ord_E$ for some $c>0$, $A_{X, D}(E)<0$, and that $-(K_{\bar X}+\bar D+E)$ is ample, $(\bar X,\bar D+E)$ is plt. 
    In this case, $-A_{X,D}(v_0)=A_{X_{v_0}, D_{v_0}}(\xi_{v_0})>0$ and $E$ is an lc place of a $\Q$-complement $\Gamma$ of $(\bar X, \bar D+E)$.
\end{lem}
\begin{proof}
    Up to scaling, we may assume $v_0=\ord_F$ for some prime divisor $F$ at infinity. 
    Given a compactification $(\bar X, \bar D+E)$ as described, and $v_0=\ord_E$, we have $-A_{X, D}(E)E\sim_\Q -(K_{\bar X}+E+\bar D)$ is ample, and so $E$ is ample.
    By Kawamata--Viehweg vanishing we have
    \[
    X_{v_0} = \Spec \bigoplus_m\frac{H^0(\bar X, mE)}{H^0(\bar X, (m-1)E)}\cong \Spec \bigoplus_mH^0(E, mE|_{E}).
    \]
    Note also that 
    \[
    A_{X,D}(E)E|_E\sim_\Q(K_{\bar X}+\bar D+E)|_E\sim_\Q K_E+\Delta_E+\bar D|_E, 
    \]
    where $\Delta_E=\mathrm{Diff}_E(0)$.
    We have $A_{X_{v_0},D_{v_0}}(\xi_{v_0})=-A_{X,D}(E)>0$ and $(X_{v_0}, D_{v_0})$ is klt. 
    
    Conversely, suppose $v_0$ is finitely generated, we can define $\bar X = \Proj\bigoplus_m \cF_{v_0}^m R$, where $E=\bar X\setminus X=\Proj \gr_{v_0}R$ is the unique ample $\Q$-Cartier prime divisor at infinity with $v_0=\ord_E$. In particular, $\bar X$ is normal by~\cite[Theorem~B]{JW26}, and $K_{\bar X}+E$ is $\Q$-Cartier. 
    If we put $\tilde X=\Spec \bigoplus_m\cF^m_{v_0}Rt^m$ and $\tilde D=C(\bar D)$ is the divisor induced by $\bar D$, then the Rees construction identifies $(X_{v_0}, D_{v_0})$ as the central fiber. Away from the cone point $x$, $\pi: \tilde X\setminus\{x\}\to \bar X$ is a Seifert $\G_m$-bundle, and we have    
    $K_{\tilde X\setminus\{x\}}|_{\tilde X\setminus\{x\}}=\pi^*(K_{\bar X})$. Since $(X_{v_0}, D_{v_0})$ is klt, we have $(\tilde X, \tilde D+X_{v_0})$ is plt near $X_{v_0}$ by inversion of adjunction. Thus $(\tilde X, \tilde D+X_{v_0})$ is plt since $(X, D)$ is klt. By~\cite[Lemma~3.1]{SingMMP}, this shows $(\bar X, \bar D+E)$ is plt and $-(K_{\bar X}+\bar D+E)\sim_\Q rE$ for some $r>0$, hence ample.  
    
    Finally, to see the last assertion, we can take $\Gamma\in |-(K_{\bar X}+\bar D+E)|_\Q$ to be a general effective $\Q$-divisor. 
\end{proof}
In general, finitely generated quasimonomial valuations can be approximated by divisorial ones.
We have the following analog of~\cite[Lemma 2.10]{LX18} (see also~\cite[Lemma 4.4]{LXZ}). 
\begin{lem}\label{lem:qmapprox}
    Let $v$ be a finitely generated negative quasimonomial valuation. Let $(Y, E)$ be a log smooth model of $X$ such that $v\in \QM_\eta(Y, E)$ for some codimension $q$ point $\eta\in Y$, where $q$ is the rational rank of $v$. Then for all valuations $w\in \QM_\eta(Y, E)$ sufficiently close to $v$, there is an isomorphism $\gr_vR\cong \gr_{w}R$ sending generators to generators. 
\end{lem}

\begin{proof}
    We may assume $\mu\colon Y\to \bar X$ is a model over a normal compactification $\bar X$ of $X$, such that $\partial \bar X$ supports an ample divisor $D$. After possible further blowups, we may also assume $E\subset Y\setminus\mu^{-1}(X)$, and that $Y$ dominates a log resolution of $(\bar X, D)$.
    Pick a regular system $y_1,\dots,y_q$ of parameters of $\cO_{Y,\eta}$ with respect to which $v=v_\alpha$ for some $\alpha=(\alpha_1, \dots, \alpha_q)\in\R_+^{q}$, where $\a_1,\dots,\a_q$ are rationally independent. 
    Then pick a finite subset $S\subset R\setminus k$ with the following property: for each point $\xi\in Y\setminus \mu^{-1}(X)$, and each irreducible component $E$ of $Y\setminus \mu^{-1}(X)$ containing $\xi$, there exist $f,g\in S$ such that $f/g\in\cO_{Y,\xi}$ and $\div(f/g)=E$ at $\xi$. 
    
    As $v$ is of finite type, there exist generators $f_1, \ldots, f_m\in R\setminus k$, whose images $\bar f_1, \ldots, \bar f_m\in \gr_vR$ generate $\gr_vR$ as a $k$-algebra. 
    For later purposes we assume $S\subset\{f_1,\dots,f_m\}$. 
    Then we have a surjection $\pi_v\colon k[x_1, \dots, x_m]\to \gr_vR, x_i\mapsto \bar f_i$. Here $k[x_1,\dots,x_m]$ is equipped with the monomial grading where $x_i$ has weight $v(f_i)$. The surjection $\pi_v$ lifts to a surjection $\pi\colon k[x_1, \dots, x_m]\to R, x_i\mapsto f_i$, and if $u$ is the monomial valuation on $k[x_1,\dots,x_m]$ with $u(x_i)=v(f_i)<0$,
    then $v$ equals the residue seminorm of $u$ under $\pi$, that is, $v(f):=\sup\{u(\phi)\mid\pi(\phi)=f\}$ for $f\in R\setminus\{0\}$. 
    Moreover, if $\Phi$ denotes the (negative of the) common value semigroup of $u$ and $v$, then $\pi$ induces a surjection of extended Rees algebras
    \[
    \bigoplus_{\la \in \Phi} t^\la\cF_u^\la k[x_1, \dots, x_m]=:\cR_u\to \cR_{v} := \bigoplus_{\la\in \Phi} t^\la \cF^\la_vR,
    \]
    which degenerates $\ker\pi$ to $\ker\pi_v$. By flatness of the extended Rees algebras over $k[t^{\Phi}]$,  any element $\bar\phi\in \ker\pi_v$ lifts to an element $\phi\in\ker\pi$.

    Now pick $\alpha'\in \R_+^{\times q}$ with $\a'\approx\a$. Let $u'$ be the monomial valuation on $k[x_1,\dots,x_m]$ with $u'(x_i)=v_{\a'}(f_i)<0$ for all $i$, and $v'$ the residue norm on $R$ induced by $u'$ under $\pi$. Then $v'$ is finitely generated and we have a surjection 
    $\pi_{v'}\colon k[x_1,\dots,x_m]\to\gr_{v'}R$. We claim that $\ker(\pi_{v'})$ is independent of $\a'$ for $\a'\approx\a$, so that $\gr_{v'}R\cong\gr_vR$ 
    by sending generators to generators.
    To prove the claim, we use the fact that every element of $\ker(\pi_{v'})$ lifts to an element of $\ker(\pi)$. Let $\{\phi_j\}_j$ be a finite set of generators for $\ker(\pi)$ lifted from $\ker(\pi_v)$. It suffices to find a decomposition $\phi_j=\psi_j+\tau_j$, with the following properties: for any $\a'\approx\a$, $\psi_j$ is $u'$-homogeneous, and $u'(\tau_j)>u'(\psi_j)=u(\phi_j)$. This will prove that $\ker(\pi_{v'})$ is generated by the images of $\psi_j$ in $\gr_{u'}k[x_1,\dots,x_m]\cong k[x_1,\dots,x_m]$, and hence independent of $\a'$.
    
    The decomposition above is simply given by taking $\psi_j$ as the $u$-initial part of $\phi_j$, that is, $\psi_j$ is $u$-homogeneous, and $u(\tau_j)>u(\psi_j)=u(\phi_j)$. As $\a_1,\dots,\a_q$ are rationally independent, we can write $u(x_i)=\langle\a,\b_i\rangle$, $1\le i\le m$, for a unique $\b_i\in\Z^r$. It follows that there exists $\beta^{(j)}\in\Z^r$ such that $\sum_ir_i\b_i=\beta^{(j)}$ for all monomials $x^r$ appearing in $\psi_j$. This implies that $u'(x^r)=\langle\a',\b^{(j)}\rangle$ for all such monomials, so $\psi_j$ is $u'$-homogeneous. Moreover, if $\a'\approx\a$, then $\sum_{i=1}^mr_i\langle\a',\beta_i\rangle>\langle\a',\b^{(j)}\rangle$ for any monomial appearing in $\tau_j$, so $u'(\tau_j)>u'(\psi_j)=u'(\phi_j)$.
     
    As $\gr_{v'}R\cong\gr_vR$ is an integral domain, $v'$ is a valuation of finite type.
    We next show that $v'=v_{\a'}$, which will complete the proof. To prove the claim, we may assume that $\a'_1,\dots,\a'_q$ are rationally independent, since both $v'$ and $v_{\a'}$ depend continuously on $\a'$. Now, $\pi_v(x_i)\ne0$ implies $\pi_{v'}(x_i)\ne0$, and hence $v'(f_i)=u'(x_i)=v_{\a'}(f_i)$ for all $i$.
    Together with the fact that $\gr_{v'}R$ is generated by the images of the $f_i$, this implies that $v'(f)\le v_{\a'}(f)$ for all $f\in R$. For the converse direction, pick a lift $\phi\in k[x_1, \dots, x_m]$ of $\bar f\in \gr_{v'}R\simeq \gr_{v}R$, and write $\phi=\phi'+\tau$ where $\phi'$ is the $u'$-initial part. For each $f_i$, write $g_i\=c_{\beta_i}y^{\beta_i}$ for the monomial in the Laurent expansion so that $v(f_i)=\langle \alpha, \beta_i\rangle$. Then $v_{\a'}(f_i-g_i)> v'(f_i)=v_{\a'}(f_i)$. Since $\pi_{v}(\phi')\ne 0$, we must have $\phi'(g_1, \dots, g_m)\ne 0$.
    Similarly, $v_{\a'}(\pi(\tau))\ge v'(\pi(\tau))>v'(\pi(\phi'))$. This gives $v_{\a'}(f)= v_{\a'}(\phi'(g_1, \dots, g_m))\le v'(f)$.

\end{proof}

\begin{cor}\label{cor:fanocomp}
    Every polarized affine Calabi--Yau variety admits a Tian--Yau type compactification.
\end{cor}
\begin{proof}
    Pick $w$ a divisorial valuation as in the previous lemma with $\gr_wR\cong \gr_{v_0}R$. Then $w$ is finitely generated, and we may assume $w=\ord_E$ for some prime divisor $E$ after scaling, and $(\bar X, E)$ is given as in Lemma~\ref{lem:divfg}.
\end{proof}
We note that for $v_0$ quasimonomial but not divisorial, the infinity divisor only gives a (divisorial) valuation sufficiently close to $v_0$. 
\begin{lem}\label{lem:ldeq}
    Let $v_0$ be a finitely generated negative valuation on $X$. If $(X_{v_0}, \xi_{v_0})$ is a log Fano cone singularity then $A_X(v_0)<0$.
    In this case we have $-A_X(v_0)=A_{X_{v_0}}(\xi_{v_0})$.
\end{lem}
\begin{proof}
    If $v_0$ is divisorial, since $X_{v_0}$ is unchanged when scaling $v_0$, and the log discrepancies are both homogeneous of degree one with respect to the scaling, this is the content of Lemma~\ref{lem:divfg}.
    In general, in view of Lemma~\ref{lem:qmapprox}, $v_0$ can be approximated by a sequence of divisorial valuations $(v_i)$ in the same dual complex with $X_{v_0}\cong X_{v_i}$.  
    This yields $A_X(v_0)=\lim_iA_X(v_i)<0$ and $\xi_{v_i}\to \xi_{v_0}$ in the Reeb cone of $X_{v_0}$. Thus $-A_X(v_0)=-\lim_iA_X(v_i)=\lim_iA_{X_{v_0}}(\xi_{v_i})=A_X(\xi_{v_0})$.
\end{proof}


%
%
\subsection{Finite generation of negative valuations}
We now exploit the finite generation property of $v_0$ to get a more concrete compactification of $X$ when $v_0$ is quasimonomial. The main goal is to prove a finite generation criterion, analogous to the one obtained for klt singularities in~\cite{XZstab}. We first introduce some terminology, which are analogs of models of qdlt Fano type and Koll\'ar models over a klt singularity. 
\begin{defi}
    A pair $(\bar X, E)$ is a \emph{(q)dlt Fano type model} of $X$ if $U\=\bar X\setminus E$ is isomorphic to $X$ in codimension one,
    every irreducible component of $E$ is a negative valuation with negative log discrepancy, and that there exists an effective $\Q$-divisor $D$ on $\bar X$ such that $(\bar X, D)$ is (q)dlt, $\lfloor D\rfloor=E$ and $-(K_{\bar X}+D)$ is ample. 
    If further $(\bar X, E)$ dlt, $-(K_{\bar X}+E)$ ample, and $\bar X$ is a compactification of $X$, then we say $(\bar X, E)$ is a \emph{Koll\'ar compactification}.
\end{defi}
\begin{rmk}
    A Tian--Yau type compactification is a Koll\'ar compactification with $E$ irreducible. 
\end{rmk}

We now state the main theorem in this section, from which Theorem~A will follow.
\begin{thm}\label{thm:kltfg}
    Let $v$ be a negative valuation on $X$. The following are equivalent:
    \begin{enumerate}
        \item The associated graded $\gr_vR$ is finitely generated and $X_v=\Spec\gr_vR$ is klt;
        \item $v$ is a bounded negative valuation with $A_X(v)<0$ and there is a Tian--Yau type compactification $(\bar X, F)$ and a special $\Q$-complement $\Gamma$ of $(\bar X, F)$ with respect to some log smooth model such that $v$ is a monomial lc place of $(\bar X, F+\Gamma)$.
        \item There exists a Koll\'ar compactification $(\bar X, E)$ of $X$ such that $v\in \QM_\eta(\bar X, E)$, where $\eta$ is the generic point of the deepest stratum.
    \end{enumerate}
\end{thm}

We first note that (3)$\Rightarrow$(1) already follows from the global finite generation property, as explored in~\cite{LXZ, XuKstab}.
\begin{lem}\label{lem:kolmodelfg}
    Let $v$ be a negative quasimonomial valuation. Then $v$ is finitely generated with $X_v=\Spec\gr_vR$ having klt singularities if there exists a Koll\'ar compactification $(\bar X, E)$ such that $v\in \QM(\bar X, E)$.
\end{lem}
\begin{proof}
    If $v\in \QM(\bar X, E)$ for some Koll\'ar compactification $(\bar X, E)$, then $-(K_{\bar X}+E)\sim_\Q\sum_i(-A_X(E_i))E_i$ is ample, and so
    $(\bar X, \sum_i(1+\varepsilon A_X(E_i))E_i)$ is log Fano for any $0<\varepsilon\ll1$, and that $(\bar X, E)$ is in particular a qdlt Fano type model for $(\bar X, \sum_i(1+\varepsilon A_X(E_i))E_i)$ in the sense of~\cite[Definition~5.8]{XuKstab}, and hence by Theorem~5.20 and Theorem~5.32 in \emph{loc.cit.}, we have $v$ is finitely generated on $\bar X$ with the degeneration being klt. Restricting the degeneration to $X$, we have that $\gr_vR$ is also finitely generated by Lemma~\ref{lem:fgres} and $X_v$ is klt.
\end{proof}

To get the remaining implications, we follow closely to the strategy developed in~\cite{LXZ, XZstab}. 
We first show (1)$\Rightarrow$(2). 
\begin{lem}\label{lem:fgcomp}
    Let $v$ be a negative quasimonomial valuation. If $v$ is finitely generated with $X_v$ having klt singularities, then there exists a Tian--Yau type compactification $(\bar X, F)$ of $X$ such that $v$ is a monomial lc place of a special $\Q$-complement $\Gamma$ with respect to some log smooth model $(Y, E)$.
\end{lem}
\begin{proof}
    Let $(\bar X', F')$ be a Tian--Yau type compactification of $X$ together with a log smooth model $\pi\colon (Y, E)\to \bar X'$ dominating $\bar X'$ such that $v\in \QM_\eta(Y, E)$ for $\eta$ the generic point of some stratum $\bigcap_{j\in J}E_j$, and that the exceptional locus of $\pi$ supports a $\pi$-ample $\Q$-divisor $H$. Fix $0<\epsilon\ll 1$ such that $L\coloneqq \pi^*F'+\epsilon H$ is still ample, and let $G\in |L|_\Q$ be a general divisor whose support does not contain any stratum of $(Y, E)$. Let $\bar D=\pi_*G$. Then $\bar D\sim_\Q F'$ and therefore takes the form $\bar D=\frac 1m \{s=0\}$ for some $s\in H^0(\bar X', mF')$. Since $H^0(\bar X', mF')\cong \cF_{F'}^m R$ by construction, $\bar D$ is the closure of $D=\frac 1m\{f=0\}$ for some $f\in \cF_{F'}^mR$. Since $\gr_vR$ is finitely generated, we may choose a set of generators $f_0=f, f_1, \cdots, f_l\in R$ such that their restrictions $\bar f_0, \cdots, \bar f_l\in\gr_vR$ generate $\gr_vR$. Fix $0<\delta\ll1$ such that $(X, \delta D)$ and $(X_v, \delta D_v)$ remain klt. Then by Lemma~\ref{lem:qmapprox} there exists a neighborhood $U\subseteq \QM_\eta(Y, E)$ of $v$ such that for all $w\in U$, we have $(X_v, \delta D_v)\cong (X_w, \delta D_w)$ is klt. By Lemma~\ref{lem:divfg}, for each divisorial valuation $w\in U$, there is a Tian--Yau type compactification $(\bar X_w, \delta\bar D_w+F_w)$, where $\bar D_w$ is the closure of $D$ on $\bar X_w$, and a $\Q$-complement $\Gamma_w$ of $(\bar X_w, \delta \bar D_w+F_w)$ such that $w$ is a monomial lc place of $(\bar X_w, \delta \bar D_w+F_w+\Gamma_w)$. By~\cite{Birbdd}, we may further take $\Gamma_w$ to be an $N$-complement, where $N$ only depends on $\dim X$ and the coefficients of $(\bar X_w, \delta \bar D_w+F_w)$. In particular, as $w$ vary in $U$, we can take the same $N$ for different compactifications. We may now write $\Gamma_w=\frac 1N\{s_w=0\}$ for some $s_w\in H^0(\bar X_w, -N(K_{\bar X_w}+F_w+\delta \bar D_w))=H^0(\bar X_w, -NA_{X, \delta D}(F_w)F_w)$, and $\Gamma_w|_X=\frac 1N (f_w=0)$ for $f_w=s_w/s_{F_w}^{-NA_{X, \delta D}(F_w)}\in R$. With this notation, for any $v\in X^\val_\infty$, using $A_{X}(v)=A_{\bar X}(v)-(1-A_{X}(F))v(F)$,
    we have
    \begin{align*}
        0&\le A_{\bar X_w, F_w+\delta \bar D_w+\Gamma_w}(v)=A_{X, \delta D}(v)+(1-A_{X, \delta D}(F_w))v(F_w)-v(F_w)-\frac 1N v(s_w) \\
    &=A_{X, \delta D}(v)-\frac 1N v(f_w),
    \end{align*} 
    with equality when $v=w$. In particular, we may choose $M=M(v)\gg0$ so that $f_w\in \cF^M_vR$ for all $w\in U$. Upon choosing a basis diagonalizing $v$, we see that the value $v(\Gamma_w|_X)$ is only determined by finitely many functions in $\cF^M_vR$ and hence finitely many Laurent monomials at $\eta$.
    Arguing as in~\cite[Lemma 4.3]{LXZ} (see also~\cite[Lemma 4.3]{XZstab}), up to replacing $U$ by a smaller neighborhood $U'$, we may pick a divisorial valuation $w\in U'$ such that $v$ is an lc place of $(\bar X_w, F_w+\Gamma)$, where $\Gamma=\Gamma_w+\delta \bar D_w$ is the desired special $\Q$-complement. Indeed, note that if $i: \bar X'\dashrightarrow \bar X_w$ is the canonical rational map, then $i|_X$ is an isomorphism preserving $D$ and hence for $p=i\circ \pi: Y\dashrightarrow\bar X$, we have $p^{-1}_*\Gamma\ge \delta G$ as desired.       
\end{proof}

It remains to show (2)$\Rightarrow$~(3), that is, to construct a qdlt Fano type compactification for a monomial lc place of a special complement. We note that the major difference in our setting is that the special complements are only taken with respect to a rational map. As we shall see, we will use the Calabi--Yau condition crucially in our proof.
The next two results give an analog of~\cite[Theorem 3.14]{XZstab}, with the first lemma serving as a tie-breaking step.
\begin{lem}\label{lem:tiebreak}
    Let $(Y, E)$ be a log smooth model of $X$, and $\Gamma_{\bar X}$ be a special $\Q$-complement on a Tian--Yau type compactification $(\bar X, F_0)$ of $X$ with respect to $(Y, E)$. Let $F_1, F_2, \cdots, F_r$ be divisors over $Y$ such that $v_i=\ord_{F_i}$ span an $(r-1)$-simplex in $\QM(Y, E)$, and that they are common lc places of $(\bar X, F+\Gamma_{\bar X})$. Then there exists a toroidal model $\pi: Z\to Y$, a sub-boundary $\Delta_Z$ and $\Gamma_Z\ge 0$ such that 
    \begin{enumerate}
        \item $\pi$ extracts exactly the divisors $F_i$ which are exceptional over $Y$;
        \item $\Gamma_Z\ge G$ for some ample divisor $G$ on $Z$ whose support does not contain any stratum of $(Z, F=\sum_iF_i)$;
        \item $K_Z+\Delta_Z+\Gamma_Z\sim_\Q0$;
        \item $\QM(Z, F)=\mathrm{LC}(Z, \Delta_Z +\Gamma_{Z})$.
    \end{enumerate}
\end{lem}
\begin{proof}
    Pick a log smooth model $W$ dominating both $Y$ and $\bar X$. Denote by $p: W\to Y$ and $q: W\to \bar X$ the projections. Define $\Gamma_Y=p_*q^*(F_0+\Gamma_{\bar X})$ and let $\Delta_Y$ be such that $K_Y+\Delta_Y=p_*q^*K_{\bar X}$. Then $(Y, \Delta_Y)$ is sub-klt, with $\lfloor \Delta_Y\rfloor\le 0$ is a sub-boundary, with support contained in $E$ and the exceptional locus over $X$. By~\cite[Lemma 3.16]{XZstab}, there is a birational toroidal morphism $\pi: Z\to Y$ extracting exactly $F_1, \cdots, F_r$ which are exceptional over $Y$ such that $-\sum_{i=1}^rA_{Y, \Delta_Y}(F_i)F_i$ is $\pi$-ample. Let $\Gamma_1\coloneqq \pi^*\Gamma_Y$ and let $K_Z+\Delta_Z = \pi^*(K_Y+\Delta_Y)$ be the crepant pullback. By construction, we have $K_Y+\Delta_Y+\Gamma_Y\sim_\Q 0 \sim_\Q K_Z+\Delta_Z+\Gamma_1\sim_\Q K_{\bar X}+F_0+\Gamma_{\bar X}$. It remains to find $\Gamma_Z\sim_\Q\Gamma_1$ satisfying the desired properties. Since $\Gamma_{\bar X}$ is a special complement with respect to $(Y, E)$, there exists an ample divisor $G_Y$ whose support does not contain any stratum of $E$ such that $D\coloneqq \Gamma_Y-G_Y\ge 0$.    
    Then the same tie-breaking procedure as in~\cite[Lemma 3.17]{XZstab} produces $\Gamma_Z\sim_\Q \Gamma_1$ such that $\Gamma_Z\ge G$ for some $G$ ample and not containing any stratum of $(Z, F)$, and $\QM(Z, F)=\mathrm{LC}(Z, \Delta_Z+\Gamma_Z)$. 

    More precisely, set $\tilde F=\sum_{i=1}^rA_{Y, \Delta_Y}(F_i)F_i$, and choose $0<\varepsilon \ll 1$ such that both $\frac 12 G_Y+\varepsilon D$ and $\frac 12\pi^*G_Y-\varepsilon \tilde F$ are ample. This implies $\pi^*(G_Y+\varepsilon D)-\varepsilon \tilde F$ is ample. Further, we may break $\Gamma_1$ as
    \[
    \Gamma_1=(1-\varepsilon)\pi^*D+\varepsilon\tilde F+(\pi^*(G_Y+\varepsilon D)-\varepsilon \tilde F).
    \]
    We claim that
    \[
    \lc(Z, \Delta_Z+(1-\varepsilon)\pi^*D+\varepsilon\tilde F)=\QM(Z, F).
    \]
    Granting the claim, we may pick $G\sim_\Q\pi^*(G_Y+\varepsilon D)-\varepsilon \tilde F$ general such that 
    \[
    \lc(Z, \Delta_Z+(1-\varepsilon)\pi^*D+\varepsilon\tilde F+G)=\QM(Z, F),
    \]
    and in particular $G$ does not contain any stratum of $(Z, F)$. We may now take $\Gamma_Z = (1-\varepsilon)\pi^*D+\varepsilon\tilde F+G$.

    To verify the claim, note that $(Z, \Delta_Z+(1-\varepsilon)\pi^*D+\varepsilon\tilde F)$ is a convex combination of $(Z, \Delta_Z+\pi^*D)$ and $(Z, \Delta_Z+\tilde F)$. The pair $(Z, \Delta_Z+\pi^*D)$ is sub-lc since $K_Z+\Delta_Z+\pi^*D = \pi^*(K_Y+\Delta_Y+D)$ and
    \[
    K_Y+\Delta_Y+D\le K_Y+\Delta_Y+\Gamma_Y.
    \]
    Moreover, by construction we have $\QM(Z, F)\subseteq \lc(Z, \Delta_Z+\pi^*D)$, since all $F_i$ are lc places of $(Y, \Delta_Y+D)$. For the paiar $(Z, \Delta_Z+\tilde F)$, note that 
    \[
    K_Z+\Delta_Z+\tilde F=\pi^*(K_Y+\Delta_Y)+\tilde F = K_Z+(\pi^{-1}_*\Delta_Y\vee \sum_{i=1}^rF_i),
    \]
    we have $(Z, \Delta_Z+\tilde F)$ is sub-lc with lc places given exactly by $\QM(Z, F)$. Combining properties of the two pairs, we have verified the claim. 
\end{proof}

\begin{thm}\label{thm:fanocomp}
    Let $v$ be a bounded negative quasimonomial valuation of rational rank $r$ with $A_X(v)<0$, which is a monomial lc place of a special complement $\Gamma$ of some Tian--Yau type compactification $(\bar X_0, F_0)$ with respect to a toroidal model $(Y', E')$. Then there is a qdlt Fano type model $(Y, E)$ of $X$ such that $v\in \QM(Y, E)$.
\end{thm}
\begin{proof}
    Pick $F_1, \dots, F_r\in \lc(\Gamma; Y', E')$ such that $\{\ord_{F_i}\}$ generate a $(r-1)$-dimensional simplex in $\lc(\Gamma; Y', E')$, and that they are all finitely generated negative valuations with negative log discrepancy.
    Let $(Z, \Delta_Z+\Gamma_Z)$ be the toroidal model as constructed in the previous lemma, and denote by $i: Z\dashrightarrow (\bar X, F_1)$ the canonical birational contraction. 
    We aim to find a birational contraction $\varphi: Z\dashrightarrow Y$ such that 
    \begin{enumerate}
        \item $\varphi$ is an isomorphism at the generic point of every stratum of $(Z, F)$;
        \item $\varphi$ contracts exactly all exceptional divisors of $i$ that are not contained in $F$.
    \end{enumerate}
    Granted this, $(Y, E=\varphi_*F)$ will give the desired qdlt Fano type compactification of $X$. Indeed, let $A$ be an effective ample $\Q$-divisor on $Y$ that is in general position, and let $A'$ be the birational transform of $A$ on $Z$. Then $\supp(A')$ doesn't contain any stratum of $F$, and hence does not contain any lc centers of $(Z, \Delta_Z+\Gamma_Z)$. Thus for $0<\varepsilon\ll 1$ rational, $G-\varepsilon A'$ is still ample, and $(Z, \Delta_Z+\Gamma_Z-G+\varepsilon A')$ is still sub-qdlt, with $\mathrm{LC}(Z, \Delta_Z+\Gamma_Z-G+\varepsilon A')=\QM(Z, F)$. By Bertini, we may choose a general effective $G_1\sim_\Q G-\varepsilon A'$ such that $(Z, \Delta_Z+\tilde\Gamma_Z)$ is sub-qdlt for 
    \[
    \tilde\Gamma_Z\coloneqq \Gamma_Z-G+\varepsilon A'+G_1,
    \]
    and $\mathrm{LC}(Z, \Delta_Z+\tilde\Gamma_Z)=\QM(Z, F)$. By construction, we also have $K_Z+\Delta_Z+\tilde\Gamma_Z\sim_\Q 0\sim_\Q K_Y+\varphi_*(\Delta_Y+\tilde \Gamma_Z)$. By property~(2), $\varphi$ in particular contracts all divisors in $\Delta_Z+\tilde \Gamma_Z$ with nonpositive coefficients, and hence the birational transform $\varphi_*(\Delta_Z+\tilde\Gamma_Z)\ge 0$ on $Y$; and~(1) implies that $(Y, \varphi_*(\Delta_Z+\tilde\Gamma_Z))$ is qdlt with lc places given by $\QM(Y, E)$. Moreover, we have constructed $\tilde\Gamma_Z$ such that $\varphi_*\tilde\Gamma_Z\ge \varepsilon A$. Therefore, if we set $D\coloneqq \varphi_*(\Delta_Z+\tilde\Gamma_Z)-\varepsilon A$, then $(Y, D)$ is qdlt with $-(K_Y+D)\sim_\Q\varepsilon A$ ample, and $\lfloor D\rfloor =E$. This gives the qdlt Fano type model of $X$.
    
    It remains to find the contraction satisfying both conditions above. To this end, note that $(Z, \Delta_Z+\Gamma_Z)$ is sub-qdlt. For some fixed $0<c\ll 1$ for which $G+cF$ remains ample, we may find $0\le G'\sim_\Q G+cF$ not containing any stratum of $F$. Put $\Gamma_Z'\coloneqq \Gamma-cF-G+G'\sim_\Q\Gamma_Z$. Then $(Z, \Delta_Z+\Gamma_Z')$ is sub-klt. Let $M$ be the sum of exceptional divisors of $i$ which are not contained in $F$.
    Moreover, if we write
    \[
    0\sim_\Q K_Z+\Delta_Z+\Gamma_Z' = K_Z+\Delta_1-\Delta_2,
    \]
    where $\Delta_1, \Delta_2$ are both effective divisors with no common components, then $\supp(\Delta_2)\subseteq M$ since $F_i$ are all lc places of $(Z, \Delta_Z+\Gamma_Z)$. We now aim to run a $(K_Z+\Delta_1+\varepsilon M)$-MMP $Z\dashrightarrow Y$ for $0<\varepsilon \ll1$ to contract all of $M$. 
    
    Indeed, arguing as in~\cite[Lemma 3.15]{XZstab}, if $(Z, \Delta_1)$ is klt, then pick $\varepsilon\ll 1$ such that $(Z, \Delta_1+\varepsilon M)$ remains klt. Since $\Delta_1\ge G$ is big, we can run the $(K_Z+\Delta_1+\varepsilon M)$-MMP $\varphi: Z\dashrightarrow Z'$ by~\cite[Theorem 1.2]{BCHM}. Since $K_Z+\Delta_1+\varepsilon M\sim_\Q \Delta_2+\varepsilon M$, and is supported on the exceptional locus of $i: Z\dashrightarrow \bar X$ by construction, we have $\varphi_*(K_Z+\Delta_1+\varepsilon M)$ is nef on $Y$ and in particular its pullback to a common resolution of $Y$ and $\bar X$ is nef over $\bar X$, and hence the MMP process will contract $M$ by the negativity lemma~\cite[Lemma 3.39]{KollarMori}.

    In general, after taking a dlt modification by~\cite{OX} of $(Z, \Delta_1)$, we get $\rho: W\to Z$ with exceptional divisors $P_i$ such that $(W, \rho^{-1}_*\Delta_1+\sum P_i)$ is dlt and $K_W+\rho^{-1}_*\Delta_1+\sum P_i$ is $\rho$-nef. By the negativity lemma, this implies $A_{Z, \Delta_1}(P_i)\le 0$ and $\rho$ is an isomorphism over the klt locus of $(Z, \Delta_1)$. In particular, since $(Z, F+M)$ is toroidal, and $\supp(\Delta_2)\subseteq M$, the generic point $\eta\notin \supp(\Delta_2)$ for any stratum of $F$, and by construction $(Z, \Delta_1)$ is klt at $\eta$. This shows $\rho$ is an isomorphism at the generic point of every stratum of $(Z, F)$. We now run an MMP on $W$ instead. Write
    \[
    0\sim_\Q\rho^*(K_Z+\Delta_Z+\Gamma_Z') = K_W+D_1-D_2,
    \]
    for $D_1, D_2$ both effective and having no common components. Let $\tilde M = \rho^{-1}_*(M)+\sum P_i$. Then $\supp(D_2)\subseteq \tilde M$. Moreover, since $(Z, \Delta_1-\Delta_2)$ is sub-klt, we have 
    \[
    D_1\le \rho_*^{-1}\Delta_1+(1-\delta) \sum P_i
    \]
    for $0<\delta \ll 1$. Note that $(W, \rho^{-1}_*\Delta_1+\sum P_i)$ is dlt, we have $(W, \rho^{-1}_*\Delta_1+(1-\delta)\sum P_i)$ is klt, and hence $(W, D_1)$ is klt, and so is $(W, D_1+\varepsilon\tilde M)$ for $0<\varepsilon\ll1$. Moreover, 
    \[
    K_W+D_1+\varepsilon \tilde M\sim_\Q D_2+\varepsilon \tilde M.
    \]
    Since $D_1+\varepsilon \tilde M$ may not be big anymore, we cannot directly apply the previous argument on $W$. Nevertheless, pick $H$ to be an ample divisor on $W$ in general position, and we claim that the $(K_W+D_1+\varepsilon \tilde M)$-MMP with scaling of $H$ will contract all of $\tilde M$ after finitely many steps. Indeed, if $\la_i$ are the scaling numbers and put $\la=\lim \la_i$, then by~\cite{BCHM} we can assume $\la=0$, and so
    after finitely many steps, we may assume that we get a contraction $\psi: W\dashrightarrow Y$ such that $D'\coloneqq\psi_*(K_W+D_1+\varepsilon\tilde M)$ is movable. The only divisors that are contracted by $\psi$ have support in $\tilde M$. Pick a common resolution $W'$ of $Y\dashrightarrow \bar X$ and denote by $p: W'\to Y$ and $q:W'\to \bar X$ the corresponding projection maps. Since $i: Z\dashrightarrow \bar X$ is a birational contraction, we have that $Y\dashrightarrow \bar X$ is also a birational contraction, since $F$ is not contained in $\tilde M$. This implies that $p^*D'$ has support in the exceptional locus of $q$. Further, since $D'$ is movable, it is nef on very general curves, and so is $p^*D'$. We then conclude that $D'=0$ by a more general version of the negativity lemma, see~\cite[Lemma 3.3]{BirkarNeg}. Let $\varphi: Z\dashrightarrow Y$ be the resulting contraction. Since $\tilde M$ does not contain any $\rho^{-1}_*F$ and $\rho$ is an isomorphism at the generic point of every stratum of $(Z, F)$, we have that $\varphi$ satisfies both~(1) and~(2).
\end{proof}



\begin{proof}[Proof of Theorem~\ref{thm:kltfg}]
    The implications (1)$\Rightarrow$(2) and (3)$\Rightarrow$(1) follow from Lemma~\ref{lem:fgcomp} and Lemma~\ref{lem:kolmodelfg}.
    For (2)$\Rightarrow$(3), the construction of a Koll\'ar compactification is achieved by choosing $F_1, F_2, \dots, F_r$ such that they generate the lattice structure on the log smooth model $\QM(Y, D)$ containing $v_0$. Then $(Z, F)$ constructed in Lemma~\ref{lem:tiebreak} is dlt at the generic points of stratum of $F$, and the resulting qdlt Fano type model $(Y', E'=\sum_iE_i')$ in Theorem~\ref{thm:fanocomp} we get is in fact dlt. Taking the ample model of $-(K_{Y'}+E')\sim_\Q -\sum_i A_X(E_i')E_i'\ge 0$ as in~\cite[Proposition~3.9]{XZstab} yields $\pi: (Y', E')\dashrightarrow (\bar X, E)$ such that $\pi$ is an isomorphism at the generic points of strata of $E$, and that $X'\=\bar X\setminus E$ is affine, since $-\sum_i A_X(E_i)E_i$ is ample by construction, and $A_X(E_i)<0$ for all $i$. Further, since $U\=Y'\setminus E'$ is isomorphic to $X$ in codimension one, we have $X$ and $X'$ are isomorphic in codimension one, and hence isomorphic. Thus $(\bar X, E)$ is the desired Koll\'ar compactification.
\end{proof}
%
%
\section{Relative volumes and valuative invariants}\label{sec:valinv}
In this section, we study relative volumes between two norms. Specifying to valuations, they define valuative invariants that will be later used to study K-stability.
%
%
\subsection{Bivariate distribution function}
We first refine the constructions of multivariate volumes in~\cite{JW26}. Consider a bounded negative norm $\chi_0$ and a bounded (but not necessarily negative) norm $\chi$ on $R$, and define 
$H=H_{\chi_0,\chi}\colon\R_+^\times\times\R\to[0,+\infty]$ by 
\[
  H(x,y)\=\limsup_\la\frac{n!}{\la^n} \dim (\cF_{\chi_0}^{\la x}R\cap\cF_\chi^{\la y}R).
\]
We call $H=H_{\chi_0,\chi}$ the \emph{bivariate distribution function} associated to $\chi_0$ and $\chi$.
\begin{prop}\label{prop:twonorms}
    The limsup defining $H(x,y)$ is a finite limit. Moreover:
    \begin{itemize}
    \item[(i)] $H(cx,cy)=c^nH(x,y)$ for $c>0$;
    \item[(ii)] $x\mapsto H(x,y)$ and $y\mapsto H(x,y)$ are increasing;
    \item[(iii)] there exists $\eta\in\R$ such that $H>0$ and $H^{1/n}$ is concave on the set $\{y>\eta x\}$, whereas $H=0$ on $\{y<\eta x\}$;
    \item[(iv)] $H(x,y)=x^n\vol(\chi_0)$ for $y\ge Cx$;
    \item[(v)] if $\chi$ is a bounded negative norm, then $H(x,y)=\vol(\frac{\chi_0}{x}\wedge\frac{\chi}{y})$ for $y>0$.
    \end{itemize}
    Here $C\ge1$ is any constant such that $C\chi_0\le\chi\le-C\chi_0$.
\end{prop}

\begin{proof}
    The assertions in~(i),~(ii),~(iv) and~(v) are clear. We must prove that the limsup is a limit, and that~(iii) holds.
    
    As $H_{\chi_0, \chi}(x,y)=H_{c\chi_0, c\chi}(cx,cy)$ for $c>0$, we may assume there exists a geometric norm $\chi_{\bar{X},D}$ dominating $\chi_0$. As in the proof of~\cite[Proposition~7.1]{JW26}, it suffices to consider integers $\la=:m$. Then $H(x,y)$ is the volume of the Okounkov body $\Delta^{x,y}$ associated to the graded subalgebra $V_m\coloneqq\cF_{\chi_0}^{mx}R\cap\cF_\chi^{my}R$ of $R(\bar X,D)$ defined by a real rank $n$ valuation on $R$.
    
    Set $\eta=\inf\{c\in\R\mid H(1,c)>0\}$, and $\Omega:=\{y>\eta x\}$. Then $H>0$ on $\Omega$, and if $(x,y),(x',y')\in\Omega$, then $\Delta^{x,y}$ and  $\Delta^{x',y'}$ are nonempty, so $\Delta^{x,y}+\Delta^{x',y'}\subset\Delta^{x+x',y+y'}$, which by the 
     Brunn--Minkowski inequality implies $H^{1/n}$ is concave, and hence continuous, on $\Omega$.
\end{proof}

%
%
\subsection{Relative volumes}\label{sec:relvol}
Let $\chi_0$ be a bounded negative norm, and $\chi$ a bounded norm on $R$. Recall the notation $T(\chi_0; \chi)=\sup_{f\in R\setminus k}\frac{\chi(f)}{\chi_0(f)}\in [0, \infty)$.

For each $\la>0$ we can restrict $\chi$ to $\cF^\la_{\chi_0}R$.  
We define
\[
\vol_\la (\chi_0; \chi) = \frac{1}{\la N_\la}\sum_{i=1}^{N_\la} \chi(f_i),
\]
where $\{f_i\}$ is a $\chi$-orthogonal basis of $\cF^\la_{\chi_0}R$, and $N_\la=\dim_k\cF^\la_{\chi_0}R$.


\begin{lem}\label{lem:Sinv}
The limit $\vol(\chi_0;\chi)\=\lim_{\la\to \infty}\vol_\la(\chi_0; \chi)$ exists, and 
    \[
    \vol(\chi_0;\chi)= \frac{1}{\vol(\chi_0)}\int_{-\infty}^{T(\chi_0; \chi)} H(1,y)\, dy-T(\chi_0; \chi)
    \]
\end{lem}
\begin{proof}
  Pick $C$ such that $T(\chi_0; \chi)\chi_0\le \chi\le -C\chi_0$, and write $T\coloneqq T(\chi_0; \chi)$.
  For $m\in \Z_{>0}$, using $H_m(1,y)\coloneqq\frac{n!}{m^n}\dim\cF_{\chi_0}^{m}R\cap \cF_\chi^{my}R$ for $-C\le y<T$, we have
  \begin{align*}
      \tfrac{n!}{m^{n+1}}mN_m\vol_m(\chi_0; \chi) 
      &= \int_{-C}^T(-y)\,\frac{d}{dy}H_m(1,y)
      =-\tfrac{n!N_mT}{m^n}+\int_{-C}^TH_m(1,y)\,dy,
  \end{align*}
  where the second equality follows from integration by parts.
    For each $y$, by Proposition~\ref{prop:twonorms}, we have $\lim_{y\to\infty}H_m(1,y)=H(1,y)$, and we have $H_m(1, y)\le H_m(1,T)=\frac{n!}{m^n}N_m\to\vol(\chi_0)$, so by Dominated Convergence,
    \[
    \frac{n!}{m^{n+1}}mN_m\vol_m(\chi_0; \chi)
    \to\int_{-C}^TH(1,y)\,dy-T\vol(\chi_0),
    \]
    which completes the proof.
\end{proof}

\begin{lem}\label{lem:relvolcts}
    Consider a bounded negative norm $\chi_0$, and a sequence $\{\chi_i\}_i$ of bounded negative norms such that $d_\infty(\chi_i,\chi_0)\to0$. Then, for any bounded norm $\chi$, we have  
    \[
    T(\chi_0;\chi)=\lim_iT(\chi_i;\chi)
    \quad\text{and}\quad
    \vol(\chi_0;\chi)=\lim_{i\to \infty}\vol(\chi_i;\chi).
    \]
\end{lem}
\begin{proof}
    Set $\d_i:=d_\infty(\chi_i,\chi_0)$, and write $T_i:=T(\chi_i;\chi)$, $T_0:=T(\chi_0;\chi)$. It is then clear that $e^{-\d_i}T_0\le T_i\le e^{\d_i}T_0$, so $\lim_iT_i=T_0$. To prove that $\vol(\chi_i;\chi)$ it now suffices to show that 
    \[
    \widetilde\vol(\chi_i;\chi)\coloneqq \frac{1}{\vol(\chi_i)}\int_{-\infty}^{T_i}H_i(1,y)\, dy \to \frac{1}{\vol(\chi_0)}\int_{-\infty}^{T_0}H_0(1,y)\, dy =: \widetilde\vol(\chi_0; \chi),
    \]
    where $H_i=H_{\chi_i,\chi}$ and $H_0=H_{\chi_0,\chi}$. The inequalities $e^{\d_i}\chi_0\le\chi_i\le e^{-\d_i}\chi_0$ yield 
    \[
    e^{-n\d_i}\vol(\chi_0)\le\vol(\chi_i)\le e^{n\d_i}\vol(\chi_0)
    \]
    as well as  
    \[
    e^{-n\d_i}H_0(1,e^{\d_i}y)
    \le H_i(1,y)
    \le e^{n\d_i}H_0(1,e^{-\d_i}y)
    \]
    for any $y\in\R$. Integration then yields
    \[
    e^{-(2n+1)\d_i}\widetilde{\vol}(\chi_0;\chi)
    \le\widetilde{\vol}(\chi_i;\chi)
    \le e^{(2n+1)\d_i}\widetilde{\vol}(\chi_0;\chi),
    \]
    and we are done.
\end{proof}
\begin{exam}\label{exam:Striv}
    For any bounded negative norm $\chi_0$, Lemma~\ref{lem:Sinv} implies
    \begin{align*}
    \vol(\chi_0; \chi_0)&=\frac{1}{\vol(\chi_0)}\int_0^1 \vol(\chi_0\wedge\frac{\chi_0}{\la})\, d\la -1 \\
    &=\frac{1}{\vol(\chi_0)}\int_0^1 \la^n\vol(\chi_0)\, d\la-1 = -\frac{n}{n+1}.
    \end{align*}
\end{exam}

\begin{exam}\label{exam:globalS}
    If $\chi_0=\chi_{\bar X, D}$ is a geometric norm given by a compactification $(\bar X, D)$ with $D$ strictly effective and ample, then for any $f\in \cF^m_{\chi_0}R\cong H^0(\bar X, mD)$, $v(f)=v(s)-mv(D)$ for $s=fs_D^m$. A direct computation shows $\vol_m(v_0; v) = S_m(\cO_{\bar X}(D); v)-v(D)$, where $S_m(\cO_{\bar X(D)}; v)$ is the usual $S$-invariant on $(\bar X, \cO_{\bar X}(D))$ as defined in~\cite{BlJ}.
\end{exam}


\begin{cor}
    For $\chi_0, \chi$ as above, we have
    \[
    \vol(\chi_0; \chi)\ge -\frac{n}{n+1}T(\chi_0; \chi).
    \]
\end{cor}

\begin{proof}
    As $\chi\ge T(\chi_0;\chi)\chi_0$, the result follows from $\chi\mapsto\vol(\chi_0;\chi)$ being increasing and 1-homogeneous, together with Example~\ref{exam:Striv}.
\end{proof}

The following is a version of~\cite[Proposition~3.15]{BlJ} and~\cite[Proposition~2.3]{LX16}.
\begin{prop}
    Let $v_0$ be a bounded negative valuation and $v$ a bounded valuation centered at infinity. Then $\vol(v_0; v)=-\frac{n}{n+1}T(v_0; v)$ iff $v=cv_0$ for some $c> 0$.
\end{prop}
\begin{proof}
    It is clear that if $v=cv_0$ for some $c> 0$, then equality holds. It suffices to show the forward direction.
    First observe that we may rescale $v$ so that $T(v_0; v)=1$, that is $v_0\le v$. We will show $v_0=v$.
    Suppose to the contrary that there is some $f\in R$ such that $v_0(f)<v(f)\ne 0$. Such $f$ always exists if $v$ is non-trivial. After replacing $f$ by a higher power, we may also assume $-q\=v_0(f)<-q+1<v(f)\eqqcolon p\ne 0$.

    For fixed $m\in \R_{>0}$, let $1\le i\le \frac mq$ be positive integers.
    Pick a basis
    \[
    \{f_{i,1}, \dots, f_{i, l_i}\}\in \cF^{m-iq}_{v_0}R
    \]
    compatible with both $v_0$ and $v$. 
    Then for each $1\le i\le \frac mq$, the images of
    \[
    \{f^if_{i,k}\mid 1\le k\le l_i\}
    \]
    form a linearly independent set in $\cF_{v_0}^mR$. 
    Moreover, for each $f^if_{i,k}$ as above, $v(f^if_{i,k})-v_0(f^if_{i,k})\ge i(v(f)-v_0(f))\ge i$.
    Thus we have
    \begin{align*}
    \vol_m(v_0; v)-\vol_m(v_0; v_0)
    &\ge \frac{1}{mN_m}\sum_{j=1}^{N_{m-iq}} (v(f^if_{i,j})- v_0(f^if_{i,j}))
    \ge \frac{iN_{m-iq}}{mN_m}\\
    \end{align*}
    Pick $i\sim O(\frac{m}{2q})$ and take $m\to \infty$, we get $\vol(v_0; v)>\vol(v_0; v_0)$, a contradiction.
\end{proof}


%
%
\subsection{Relative volumes and reductions}
Let $\chi_0$ be a negative norm of finite type, and $\chi$ a seminorm. For any $\la\ge0$, the restriction $\chi|_{\cF^\la_{\chi_0}R}$ induces a residue seminorm $\gr^\la_{\chi_0}\chi$ on $\gr^\la_{\chi_0}R=\cF^\la_{\chi_0}R/\cF^{<\la}_{\chi_0}R$. For $\bar f\in\gr_{\chi_0}^\la R$ and $\bar f'\in\gr_{\chi_0}^{\la'}R$ we have $\bar f\bar f'\in\gr_{\chi_0}^{\la+\la'}R$ and $\gr_{\chi_0}^{\la+\la'}(\bar f\bar f')\ge\gr_{\chi_0}^\la(\bar f)+\gr_{\chi_0}^{\la'}(\bar f')$, so we obtain a seminorm $\gr_{\chi_0}\chi$ on $\gr_{\chi_0}R=\bigoplus_\la\gr_{\chi_0}^\la R$. This seminorm is graded in the sense that the different subspaces $\gr^\la_{\chi_0}R$ are orthogonal.

The seminorm $\gr_{\chi_0}\chi_0$, which takes the value $-\lambda$ on $\gr_{\chi_0}^\la R$ for each $\la\in\Phi_{\chi_0}$, is a negative bounded norm. Indeed, let $k[z_1,\dots,z_m]\to R$ be a surjection such that $\chi_0$ is the residue norm induced by a negative monomial valuation $u_0$. Then $\gr_{\chi_0}\chi_0$ is the residue norm of $\gr_{u_0}u_0$ on $\gr_{u_0}$ under the surjection $\gr_{u_0}k[z_1,\dots,z_m]\to\gr_{\chi_0}R$. Now $\gr_{u_0}k[z_1,\dots,z_m]$ is a negative monomial valuation on $\gr_{u_0}k[z_1,\dots,z_m]\simeq k[z_1,\dots,z_m]$, so $\gr_{\chi_0}\chi_0$ is a bounded negative norm.
It follows that if $\chi$ is a bounded (resp.\ bounded negative) norm, so is $\gr_{\chi_0}\chi$.

The grading on $\gr_{\chi_0}R$ gives rise to a translation action by $\R$. If $\rho$ is a graded seminorm on $\gr_{\chi_0}R$, and $c\in\R$, then there is a unique graded seminorm $\rho+c$ on $\gr_{\chi_0}R$ such that $\rho+c=\rho+c\la$ on $\gr_{\chi_0}^\la$. (For example, we can apply this to $\rho=\gr_{\chi_0}\chi$.) If $\rho$ is bounded, so is $\rho+c$. If $\rho$ is bounded negative, so is $\rho+c$ for $c\le 0$.

As before, let $\chi_0$ be a bounded negative norm, and $\chi$ a bounded norm on $R$. Then $\chi$ induces a graded norm $\gr_{\chi_0}\chi$ on $\gr_{\chi_0}R$. For any $c\in\R$, we also get a shifted graded norm $\gr_{\chi_0}\chi+c$ on $\gr_{\chi_0}R$ by setting 
\[
(\gr_{\chi_0}\chi+c)(\bar f) = \gr_{\chi_0}\chi(\bar f)+c\la \text{ for any } \bar f\in \gr_{\chi_0}^\la R.
\]

\begin{lem}\label{lem:gradedvol}
    We have $\vol(\gr_{\chi_0}\chi_0;\gr_{\chi_0}\chi+c)=\vol(\chi_0;\chi)+\frac{cn}{n+1}$.
\end{lem}
\begin{proof}
    For each $\la\ge 0$, we have 
    \[
    N_\la \= \dim\cF^\la_{\chi_0} R = \sum_{\mu\le\la}\dim\gr_{\chi_0}^{\mu}R = \dim \cF^\la_{\gr_{\chi_0}\chi_0}\gr_{\chi_0}R.
    \]
    Pick a common orthogonal basis $\{f_j\}_0^{N_\la}$ of $\cF^\la R$ for $\chi_0$ and $\chi$, and $\la\gg0$. We can choose this basis such that $f_0=1$ and $\chi_0(f_j)\ge\chi_0(f_{j+1})$ for all $j$. Let $0=\la_0<\la_1<\la_2<\dots$ be the discrete subset of $\R_{\ge0}$ consisting of jumping numbers for $\chi_0$, that is, numbers $\la\ge0$ such that $\gr^\la_{\chi_0}R\ne0$. Set $N_j:=\dim\cF_{\chi_0}^{\la_j}R$ and $N'_j:=\dim\gr^{\la_j}_{\chi_0}R$. Then $N_0=N'_0=1$, and $N_j=N_{j-1}+N'_j$ for $j\ge1$. Moreover, $\chi_0(f_0)=0$, whereas $\chi_0(f_l)=-\la_j$ for $N_{j+1}<l\le N_j$, $j\ge 1$. Since the basis is orthogonal for both $\chi$ and $\chi_0$ implies that the images $\bar f_l$ of $f_l$ in $\gr^{\la_j}_{\chi_0}R$, $N_{j-1}<l\le N_j$, form an orthogonal basis for $\gr^{\la_j}_{\chi_0}\chi$, with $\gr^{\la_j}_{\chi_0}\chi(\bar f_l)=\chi(f_l)$.   
    It now follows that 
    \begin{align*}
    \vol(\chi|_{\cF^{\la_j}_{\chi_0}R})-c\vol(\chi_0|_{\cF^{\la_j}_{\chi_0}R})
    &=\sum_{0\le l\le N_j}(\chi(f_l)-c\chi_0(f_l)) \\
    &=\sum_{1\le k\le j}\sum_{N_{k-1}<l\le N_{k}}(\gr_{\chi_0}^{\la_k}\chi(\bar f_l)+c\la_k) \\
    &=\vol(\gr_{\chi_0}\chi|_{\cF^{\la_j}_{\gr_{\chi_0}\chi_0}\gr_{\chi_0}R}+c).
    \end{align*}
    Thus
    \[
    \vol_{\la_j}(\chi_0; \chi)-c\vol_{\la_j}(\chi_0; \chi_0) = \vol_{\la_j}(\gr_{\chi_0}\chi_0; \gr_{\chi_0}\chi+c),
    \]
    and the result follows by letting $\la_j\to \infty$ and applying Example~\ref{exam:Striv}.
\end{proof}

%
%
\section{K-semistable negative valuations}\label{sec:Kss}
We now assume that $X$ is a klt affine Calabi--Yau variety with $v_0$ a bounded negative quasimonomial valuation. The goal is to define K-semistability of $v_0$ through valuative invariants and prove Theorems~B and~C.

%
%
\subsection{Valuative invariants}
We first define K-semistability of $v_0$ through valuative invariants. 
In what follows, we fix a bounded negative quasimonomial valuation $v_0$ on $X$ with $A_X(v_0)<0$. Let $v\in X^\val_\infty$ be of linear growth. 
Recall that 
\[T(v_0; v)=\FS(\chi_{v_0})(v)=\sup_{f\in R\setminus k}\frac{v(f)}{v_0(f)}.\]
We define
\[
S(v_0; v)\coloneqq A_X(v_0)\frac{\vol(v_0; v)}{\vol(v_0; v_0)} = -\frac{(n+1)A_X(v_0)}{n}\vol(v_0;v).
\]
Without assuming the finite generation property, we use a valuative criterion to define K-semistability for negative quasimonomial valuations.
\begin{defi}
    A bounded negative quasimonomial valuation $v_0$ on $X$ is said to be \emph{K-semistable} if $A_X(v_0)<0$ and
    \[
    A_X(v)\ge S(v_0; v),
    \]
    for all $v\in X^\val_\infty$. 
    If $v_0$ is further finitely generated with klt degenerations, then we say that the polarized affine Calabi--Yau variety $(X, v_0)$ is K-semistable.
\end{defi}
Analogous to the definition of $\delta$-invariants for Fano varieties, see~\cite{BlJ}, for a bounded negative quasimonomial valuation $v_0$ on $X$ with $A_X(v_0)<0$, we define
\[
\delta(X,v_0)\coloneqq \inf\left\{\frac{A_X(v)-A_X(v_0)T(v_0; v)}{S(v_0; v)-A_X(v_0)T(v_0; v)} \right\},
\]
where infimum is taken over all $v\in X^\val_\infty$ of finite log discrepancy and $v\notin \R_{>0}v_0$.
It follows directly from the definition that $v_0$ is K-semistable iff $\delta(X,v_0)\ge 1$.


For $v_0$ finitely generated, using Lemma~\ref{lem:gradedvol}, we also have the following comparison of valuative invariants between the ones on $X$ and on $(X_{v_0},\Delta_{v_0}, \xi_{v_0})$.
\begin{lem}\label{lem:STcomp}
   Let $\chi$ be a bounded norm on $R$ and set $\widehat{\gr_{v_0}\chi}\coloneqq \gr_{v_0}\chi+T(v_0; \chi)$. Then:
   \begin{itemize}
       \item[(i)] $S(v_0; \chi)-A_X(v_0)T(v_0; \chi)=S(\xi_{v_0}; \widehat{\gr_{v_0} \chi}).$
       \item[(ii)] $T(\xi_{v_0}; \widehat{\gr_{v_0} \chi})=T(v_0; \chi)-T(\chi;v_0)^{-1}$ if $\chi$ is negative, and $T(\xi_{v_0}; \widehat{\gr_{v_0} \chi})\ge T(v_0; \chi)$ otherwise.
   \end{itemize}
\end{lem}
\begin{proof}
    Write $r=-A_X(v_0)$. Since $\chi(f)-v_0(f)T(v_0; \chi)\ge 0$ for all $f\in R$, $\widehat{\gr_{v_0}\chi}$ defines a non-negative graded norm on $(X_{v_0}, \xi_{v_0})$. By Lemma~\ref{lem:ldeq} and Lemma~\ref{lem:gradedvol}, we get
    \begin{align*}
    S(\xi_{v_0}; \widehat{\gr_{v_0}\chi}) 
    &=\frac{(n+1)r}{n}\vol(\gr_{v_0}v_0; \gr_{v_0}\chi+T(v_0; \chi)) \\
    &=\frac{(n+1)r}{n}(\vol(v_0; \chi)+\frac{n}{n+1}T(v_0; \chi)) = S(v_0; \chi)+rT(v_0; \chi).
    \end{align*}

    To prove~(ii), note that for negative norms $\chi$, we have $v_0(f)T(v_0; \chi)\le \chi(f)\le T(\chi; v_0)^{-1}v_0(f)$ for any $f\in R\setminus k$, and it follows that $T(\xi_{v_0}, \widehat{\gr_{v_0}\chi}) = -T(\chi; v_0)^{-1}+T(v_0; v)$. If $\chi$ is not negative, then there is some $f\in R\setminus k$ with $\chi(f)\ge 0$, and so $T(\xi_{v_0}; \widehat{\gr_{v_0}\chi})\ge \chi(f)+T(v_0; \chi)\ge T(v_0; \chi)$. 
\end{proof}
We first record an analog of~\cite[Lemma~3.3]{XZstab}.
\begin{lem}\label{lem:ksslct}
    If a negative valuation $v_0$ is K-semistable, then for any effective $\Q$-Cartier $\Q$-divisor $D\=\frac 1k\{f=0\}$ with $f\in R$ on $X$ satisfying $v_0(D)>\frac{A_X(v_0)}{n}$, we have
    \[
    \lct(X, D; \chi_{v_0})=-A_{X, D}(v_0).
    \]
\end{lem}

\begin{proof}
    Up to scaling we may assume $A_X(v_0)=-1$, and put $c\coloneqq -A_{X, D}(v_0)$. As observed in Lemma~\ref{lem:lctgeld}, we always have $\lct(X, D; \chi_{v_0})\ge c$. Suppose the inequality is strict, namely, there is a valuation $v\in X^\val_\infty$ such that $-A_{X, D}(v)>c\FS(\chi_{v_0})(v)$, or equivalently, $A_X(v)<v(D)-c\FS(\chi_{v_0})(v)$. Pick $0<\varepsilon\ll 1$ such that 
    \[
    A_X(v)+\varepsilon<v(D)-c\FS(\chi_{v_0})(v).
    \]
    We will show that this is impossible if $v_0$ is K-semistable.
    The idea is to use basis type divisors to estimate $\FS(\chi_{v_0})(v)$. For each $\la>0$, we define a $(\la, v_0)$-basis type divisor to be an effective $\Q$-divisor of the form
    \[
    \Gamma\coloneqq \frac{1}{-\vol(\chi_{v_0}|_{\cF^\la_{v_0}R})}\sum_i \{f_i=0\},
    \]
    where $\{f_i\}$ is any $v_0$-orthogonal basis for $\cF^\la_{v_0}R$. For fixed $\la>0$, $v_0(\Gamma)=-1$ for any $(\la, v_0)$-basis type divisor $\Gamma$, and 
    \[
    v(\Gamma)\le -\frac{\vol(\chi_{v}|_{\cF^\la_{v_0}R})}{\vol(\chi_{v_0}|_{\cF^\la_{v_0}R})}=A_X(v_0)\frac{\vol_\la(v_0; v)}{\vol_\la(v_0; v_0)},
    \]
    with equality achieved when $\Gamma$ is also compatible with $v$. Now we may choose $\la\gg 0$ such that
    \[
    v(\Gamma)\le S(v_0; v)+\varepsilon\le A_X(v)+\varepsilon,
    \]
    where the second inequality follows from K-semistability.
    We claim that there is a $(\la, v_0)$-basis type divisor $\Gamma\ge D$. Granting the claim, write $\Gamma-D = D_1+D_2$, where $v(D_1)>0$ and $v(D_2)\le 0$, then 
    \[
    \frac{v(\Gamma-D)}{v_0(\Gamma-D)}=\frac{v(D_1)+v(D_2)}{v_0(D_1)+v_0(D_2)}\le \frac{v(D_2)}{v_0(D_2)}\le \FS(\chi_{v_0})(v),
    \]
    and
    \begin{align*}
        v(D)-c\FS(\chi_{v_0})(v)&=v(D)-(1+v_0(D))\FS(\chi_{v_0})(v) = v(D)-v_0(D-\Gamma)\FS(\chi_{v_0})(v)\\
        &\le v(D)+v_0(\Gamma-D)\frac{v(\Gamma-D)}{v_0(\Gamma-D)}=v(\Gamma)\le A_X(v)+\varepsilon.
    \end{align*}
    This gives a contradiction.

    It remains to find a basis type divisor $\Gamma\ge D$, and we may assume $D\ne 0$. 
    For each $\la>0$, define an $\N$-filtration $\cF$ on $V_\la\=\cF^\la_{v_0}R$ via
    \[
    \cF^jV_\la\coloneqq \{g\in V_\la\mid gf^j\in V_\la\}=V_{\la+jv_0(f)}.
    \]
    Pick a basis compatible with both $\cF$ and $v_0$, we get a basis type divisor $\Gamma$ satisfying $-\vol(\chi_{v_0}|_{\cF^\la_{v_0}R})\Gamma\ge \mu_\la\div(f)=k\mu_\la D$, where $\mu_\la=\sum_{j=1}^{\lfloor \frac{\la}{-v_0(f)}\rfloor}\dim \cF^jV_\la$. Note that 
    \[
    \lim_\la \frac{\mu_\la}{\la^{n+1}/(n+1)!} = \frac{\vol(v_0)}{-v_0(f)}, \quad \lim_\la \frac{-\vol(\chi_{v_0}|_{\cF^\la_{v_0}R})}{\la^{n+1}/(n+1)!}=n\vol(v_0),
    \]
    we have
    \[
    \lim_\la \frac{k\mu_\la}{-\vol(\chi_{v_0}|_{\cF^\la_{v_0}R})} = \frac{1}{-nv_0(D)}> \frac{1}{-A_X(v_0)}=1.
    \]
    Thus we have $\Gamma\ge D$ for $\la\gg 0$, and this completes the proof.
\end{proof}
\begin{rmk}\label{rmk:tylct}
    Let us mention that there are many more valuations computing its own lct. In fact, if $(\bar X, E)$ is a Tian--Yau type compactification and $v=\ord_E$, then $v$ computes $\lct(\chi_v)$:
    \[
    \lct(\chi_v)=\sup_w\frac{-A_X(w)}{w(E)} = \sup_w\frac{(1-A_X(v))w(E)-A_{\bar X}(w)}{w(E)}=1-A_X(v)-\lct(\bar X, E)=-A_X(v).
    \]
\end{rmk}

Recall that if $(X, x)$ is a klt singularity, then the minimizer of the normalized volume is the unique K-semistable valuation. When studying negative valuations, however, we do not have uniqueness in general. In fact, as observed in~\cite{GaborCY}, there are infinitely many K-semistable negative divisorial valuations on $X=\A^n$ for $n\ge 3$. Nevertheless, we will show that K-semistable negative valuations indeed give K-semistable degenerations when finitely generated. The goal is to prove the following more precise version of Theorem~B.
\begin{thm}\label{thm:qmvaluative}
    For any polarized affine Calabi--Yau variety $(X, v_0)$ we have 
    \[\min\{1,\d(X,v_0)\}\ge\d(X_{v_0},\xi_{v_0}),\]
    with equality if $v_0$ computes $\lct(\chi_{v_0})$.
    In particular, $(X,v_0)$ is K-semistable iff $(X_{v_0}, \xi_{v_0})$ is a K-semistable log Fano cone singularity.
\end{thm}

%
%
\subsection{The divisorial case}
Let $(\bar X, E)$ be an $n$-dimensional plt pair with $E$ a $\Q$-Cartier prime divisor and $-(K_{\bar X}+E)$ ample. Let $\Delta_E$ be such that $(K_{\bar X}+E)|_E=K_E+\Delta_E$, and assume $K_{\bar X}+(1+r)E\sim_\Q 0$ for some $r>0$. 

The following global result, analogous to~\cite[Theorem 3.6]{XZ21}, implies the divisorial case of Theorem~\ref{thm:qmvaluative},
see Corollary~\ref{cor:divvaluative}.

\begin{prop}\label{prop:stabadj}
    With the above setup, we have
    \[
    \delta(E,\Delta_E)
    =\frac{n}{r}\inf_v\frac{A_{\bar X,E}(v)}{(n+1)S(E;v)-v(E)},
    \]
    where the infimum is over all valuations $v\in \bar{X}^\val$ with $0<A_{\bar X,E}(v)<\infty$, and having center on $E$.
\end{prop}

\begin{proof}
    As in~\cite{AZ22} we use adjunction, and we freely borrow definitions and results from \textit{loc. cit}. Set $L:=\frac{n+1}{n}rE$ and consider the bigraded linear series $W_\bullet=\{W_{m,j}\}$ for $L|_E$, given by
    \[
    W_{m,j}:=\im(H^0(\bar X,mL-jE)\to H^0(E,mL|_E-jE|_E)).    
    \]
    Let us first prove that 
    \begin{equation}\label{equ:twodeltas1}
    \d:=\d(E,\Delta_E)=\delta(E,\Delta_E;W_\bullet).
    \end{equation}
    Here $\d(E,\Delta_E;W_\bullet)$ is the limit of $\d_m(E,\Delta_E;W_\bullet)$, the infimum of $\lct(E,\Delta_E;\Gamma)$ over all $m$-basis type $\Q$-divisors $\Gamma$ of $W_\bullet$.

    By~\cite[Example 2.12]{AZ22}, $W_\bullet$ is almost complete with respect to $E|_E$, with fixed part $F(W_\bullet)=0$, and by~(3.1) in \emph{loc.\ cit.} the divisor class of $W_\bullet$ equals
    \[
    c_1(W_\bullet)=(L-S(L; E)E)|_E\sim_\Q (L-\frac rn E)|_E\sim_\Q rE|_E\sim_\Q -(K_E+\Delta_E),
    \]
    where we have used $S(E; E)=\frac{1}{n+1}$. By~\cite[Lemma 2.13]{AZ22} we therefore have  $S(W_\bullet; w)=S(-(K_E+\Delta_E); w)$ for any linearly bounded valuation $w$ on $E$, and this implies~\eqref{equ:twodeltas1} in view of the valuative formulas for the stability thresholds, see~\cite[Lemma~2.9]{AZ22}.

    Now set $\d_m:=\delta_m(E,\Delta_E;W_\bullet)$ and $\d=\d(E,\Delta_E;W_\bullet)$. We must prove that \begin{equation}\label{equ:threshold}
      \d=\inf_v\frac{A_{\bar X,E}(v)}{S(L;v)-\frac{r}{n}v(E)}=:\kappa.
    \end{equation}

    Given any $m$-basis type $\Q$-divisor $D$ for $L$ compatible with $E$, we can write 
    \[
    D=S_m(L;E)E+\tilde{D},
    \]
    where $\tilde{D}$ is an effective $\Q$-divisor on $\bar X$ with $\ord_E(\tilde D)=0$. Then $\tilde D|_E$ is an $m$-basis type $\Q$-divisor for $W_\bullet$. Conversely, any $m$-basis $\Q$-divisor for $W_\bullet$ is of this form. Indeed, for each $m$ and $j$ with $W_{m,j}\ne 0$, we may lift a basis of $W_{m,j}$ to sections $\{s_{j,i}\}$ of $H^0(\bar X, mL-jE)$ with $\ord_E(s_{j,i})=0$, since $H^0(\bar X, kE)\twoheadrightarrow H^0(E, kE|_E)$ for any $k>0$. This gives rise to a linearly independent set $\{s_{j,i}s_E^j\}$ in $H^0(\bar X, mL)$ compatible with $E$. Taking the union over all $j$, we get a basis of $H^0(\bar X, mL)$ compatible with $E$.

This yields
    \begin{align*}
    \d_m
    &=\inf_D\lct(E,\Delta_E;\tilde D|_E)
    =\inf_D\lct_E(\bar X,E;\tilde D)\\
    &=\inf_D\inf_v\frac{A_{\bar X,E}(v)}{v(\tilde D)}
    =\inf_v\inf_D\frac{A_{\bar X,E}(v)}{v(\tilde D)},
    \end{align*}
    where the second equality follows from inversion of adjunction.
    Now, for fixed $v$, we have 
    \[
    v(\tilde D)=v(D)-S_m(L; E)v(E)\le S_m(L;v)-S_m(L;E)v(E),
    \]
    with equality if $D$ is compatible with both $E$ and $v$. 
    After interchanging the inf in $\delta_m$, it follows that 
    \[
    \d_m=\inf_v\frac{A_{\bar X,E}(v)}{S_m(L;v)-S_m(L;E)v(E)}.
    \]
    By uppersemicontinuity, we have $\delta=\limsup_m\delta_m\le \kappa$.
    On the other hand, if we write $\cF_v$ for the induced filtration on $E$, then $S_m(W_\bullet; \cF_v)\ge v(\tilde D)$ by submultiplicativity, with equality for $D$ compatible with both $E$ and $v$. Hence we have $S_m(W_\bullet; \cF_v)=\sup_D v(\tilde D)$.     
    As in~\cite{AZ22} (see~\cite[Corollary~2.10]{BlJ}) there exist $\e_m\nearrow1$ such that $S(W_\bullet;\cF_v)\ge\e_m S_m(W_\bullet;\cF_v)$ for all $v$.  
    This gives $S(L; v)-S(L; E)v(E)=S(W_\bullet; \cF_v)\ge \varepsilon_mS_m(W_\bullet; \cF_v)=\varepsilon_m(S_m(L; v)-S_m(L;E)v(E))$ for all $v$.
    Letting $m\to\infty$, we therefore get
    \[
    \d=\inf_v\frac{A_{\bar X,E}(v)}{S(L;v)-S(L;E)v(E)},
    \]    
    which equals $\kappa$, since $S(L;E)=\frac{r}{n}$.
  \end{proof}
\begin{cor}\label{cor:divvaluative}
    If $(X,v_0)$ is a polarized affine Calabi--Yau variety, with $v_0$ divisorial, then
    \begin{equation}\label{equ:twodeltas2}
    \min\{1, \d(X,v_0)\}=\d(X_{v_0}, \xi_{v_0}).
    \end{equation}
    In particular, $(X,v_0)$ is K-semistable iff $(X_{v_0}, \xi_{v_0})$ is K-semistable.
\end{cor}
\begin{proof}
    By scaling invariance, we may assume $v_0=\ord_E$ for some prime divisor $E$.
    The projective variety $\bar X = \Proj \bigoplus_m \cF_{v_0}^mR$ is a compactification of $X$, and if we set $E:=\bar X\setminus X$, then $v_0=\ord_E$, and the pair $(\bar X, E)$ satisfies the conditions before Proposition~\ref{prop:stabadj}, with $r=-A_X(v_0)$.
    The log Fano cone $(X_{v_0},\xi_{v_0})$ is the cone over the log Fano pair $(E,\Delta_E)$, and $\d(X_{v_0},\xi_{v_0})=\min\{1,\d(E,\Delta_E)\}$. It therefore suffices to show that $\d(E,\Delta_E)=\d(X,v_0)$.

    By Proposition~\ref{prop:stabadj}, we have 
    \[
    \d(E,\Delta_E)
    =\inf_v\frac{n}{r}\frac{A_{\bar X,E}(v)}{(n+1)S(E;v)-v(E)}.
    \]
    Here $A_{\bar X,E}(v)=A_X(v)+rv(E)$, whereas $S(E;v)=\vol(v_0;v)+v(E)$, see Example~\ref{exam:globalS}.
    As $A_X(v_0)=-r$, and $T(v_0,v)=v(E)$ by~\eqref{equ:FSgeom}, we further have 
    \[
    \frac{(n+1)r}{n}S(E; v)-\frac rn v(E)=S(v_0; v)+rT(v_0; v).
    \]
    So it follows that 
    \[
    \d(E,\Delta_E)=\inf_v\frac{A_X(v)-A_X(v_0)T(v_0;v)}{S(v_0;v)-A_X(v_0)T(v_0;v)}
    =\d(X,v_0),
    \]
    which completes the proof.
\end{proof}

%
%
\subsection{The general case}

We are now ready to complete the proof of Theorem~\ref{thm:qmvaluative}.
\begin{proof}[Proof of Theorem~\ref{thm:qmvaluative}]
    The divisorial case is treated in Corollary~\ref{cor:divvaluative}.
    We may assume $v_0$ is not divisorial. 
    
    We first set up some notation. By Lemma~\ref{lem:qmapprox} we can find divisorial valuations with $\de_\infty(v_i, v_0)\to 0$ with $X_{v_i}\cong X_{v_0}$ for all $i$. By construction, $X_{v_0}$ is a log Fano cone singularity with $-r_0\=A_X(v_0)<0$.
    For each $i$, pick $q_i\in \Q_{>0}$ such that $\ord_{E_i}=q_iv_i$ for some prime divisor $E_i$, and let $(\bar X_i, E_i)$ be the corresponding Tian--Yau type compactifications with $q_ir_i=-q_iA_X(v_i)=-A_X(E_i)$ and let $L_i=\frac{-(n+1)A_X(E_i)}{n}E_i$. By~\cite[Lemma~2.7]{LX18} (see also~\cite[Lemma~3.1]{LXZ}, we can choose $v_i$ and $q_i$ so that $q_i\de_\infty(v_i, v_0)\to 0$, and $q_ir_i-q_ir_0\to 0$. In particular, this implies $q_ir_0T(v_0; v_i)-q_ir_i\to 0$.
    In what follows, we will fix this sequence $v_i\to v_0$ and denote by $\xi_{v_i}$ and $\xi_{v_0}$ the Reeb fields on $X_{v_0}$ with respect to $v_i$ and $v_0$ respectively. Without loss of generality, we may pick the sequence with $v_i\ge v_0$ so that $\xi_{v_i}\le \xi_{v_0}$. 
    Set $\delta_i=\delta(E_i, \Delta_{E_i}) = \delta(X_{v_0}, \xi_{v_i})$ and $\delta=\delta(X_{v_0}, \xi_{v_0})$.

    We first show $\min\{1, \delta(X, v_0)\}\ge \delta$. Fix a valuation $v\in X^\val_\infty$ with finite log discrepancy. By Corollary~\ref{cor:divvaluative}, for each $i$, we have
    \[
    A_{\bar X_i, E_i}(v)\ge \delta_i(S(L_i; v)-\frac{r_i}{n}v(E_i)), 
    \]
    or equivalently,
    \[
    A_X(v)\ge \delta_i (S(L_i; v)-\frac{q_ir_i}{n}v(E_i))-q_ir_iv(E_i) = \delta_iS(v_i; v)+(\delta_i-1)r_iT(v_i; v).
    \]
    By Lemma~\ref{lem:FScts}, Lemma~\ref{lem:relvolcts} and Proposition~\ref{lem:deltaapprox}, we may take $i\to \infty$ and get
    \[
    A_X(v)+r_0T(v_0; v)\ge \delta(S(v_0; v)+r_0T(v_0; v)),
    \]
    which implies $\delta(X, v_0)\ge \delta$ since $\delta\le 1$.

    Conversely, we need to show $\delta \ge \delta'\=\min\{1, \delta(X, v_0)\}$ when $\lct(\chi_{v_0})=-A_X(v_0)$. Set $\delta_m = \delta_m(X_{v_0}, \xi_{v_0})$. We will show $\lim_m\delta_m\ge \delta'$. 
    As described in Lemma~\ref{lem:STcomp}, any $v\in X^\val_\infty$ induces a (shifted) filtration $\cF_v$ on $\gr_{v_0}R$ with
    \[
    S(v_0; v)+r_0T(v_0; v)=S(\xi_{v_0}; \cF_v).
    \]
    Using~\eqref{equ:Sapprox}, for $m'=\lfloor \frac{(n+1)r_0}{n}m\rfloor$,
    there exists $\varepsilon_m\nearrow 1$ such that for any $v\in X^\val_\infty\setminus\R_{>0}v_0$, we have
    \[
    \delta'\varepsilon_mS_{m'}(\xi_{v_0}; \cF_v)\le \delta'S(\xi_{v_0}; \cF_v)=\delta'(S(v_0; v)+r_0T(v_0; v))\le A_X(v)+r_0T(v_0; v),
    \]
    where the last inequality follows from the definition of $\delta'$. Note that for any $i$, we have $r_0T(v_0; v)\le \frac{r_0}{r_i}T(v_0; v_i)r_iT(v_i;v)$, using $A_X(v)+r_iT(v_i; v)=A_{\bar X_i, E_i}(v)$, we may rewrite the above inequality as
    \[
    \delta'\varepsilon_mS_{m'}(\xi_{v_0}; \cF_v)\le A_{\bar X_i, (1-t_i)E_i}(v),
    \]
    where $t_i=q_i(r_0T(v_0; v_i)-r_i)\ge 0$ 
    since $\lct(\chi_{v_0})=-A_X(v_0)$. For fixed $m'$, after taking $i$ large enough, any $(m', \xi_{v_0})$-basis type divisor $D_{m'}(\xi_{v_0})$ is a multiple of an $(m', \xi_{v_i})$-basis type divisor when $R_{m'}(\xi_{v_0})=R_{m'}(\xi_{v_i})$. More precisely, for our choices of $\xi_{v_i}$, if we write $\xi_{v_i}\le \xi_{v_0}\le c_i\xi_{v_i}$ for $1\le c_i\to 1$, then $R_{m'}(\xi_{v_0})\subseteq R_{m'}(\xi_{v_i})\subseteq R_{c_im'}(\xi_{v_0})$. Thus for $i\ge i(m)$ large enough so that $c_im'$ is not in the value group of $\xi_{v_0}$, we have $R_{m'}(\xi_{v_i})=R_{m'}(\xi_{v_0})$ and
    \[
    D_{m'}(\xi_{v_0}) = \frac{(n+1)r_0}{m'N_{m'}}\sum_{j=1}^{N_{m'}}\{f_j=0\}=\frac{r_0}{r_i}D_{m'}(\xi_{v_i}),
    \] 
    where $D_{m'}(\xi_{v_i})$ is an $(m', \xi_{v_i})$-basis type divisor.
    Since $S_{m'}(\xi_{v_i}; \cF_v)\le S_{m'}(\xi_{v_0}; \cF_v)$ can be computed using $(m', \xi_{v_0})$-basis type divisors, the above inequality implies, for the above choice of $i$, for all $(m', \xi_{v_0})$-basis type divisors $D_{m'}(\xi_{v_0})$, we have $(\bar X_i, (1-t_i)E_i+\delta'\varepsilon_m\frac{r_0}{r_i}\tilde D_{m'}(\xi_i))$ is lc near $E_i$, for any lift of $D_{m'}(\xi_{v_0})$ to an $m'$-basis type divisor of $(\bar X_i, L_i)$ compatible with $E_i$, with $\tilde D_m(\xi_{v_i})|_{E_i}=D_m(\xi_{v_i})$. Note that the coefficients of $\varepsilon_m\frac{r_0}{r_i}\tilde D_{m'}(\xi_{v_i})$ is by construction independent of $i$, by the ACC of log canonical thresholds (see~\cite{HMX}), there exists $\varepsilon>0$ only depending on $\dim \bar X_i=\dim X$ and the coefficients of $\varepsilon_m\frac{r_0}{r_i}\tilde D_{m'}(\xi_{v_i})$, such that for any reduced divisor $G$ on $\bar X_i$, if $(\bar X_i, (1-\varepsilon)G+\delta'\varepsilon_m\frac{r_0}{r_i}\tilde D_{m'}(\xi_{v_i}))$ is lc, then so is $(\bar X_i, G+\delta'\varepsilon_m\frac{r_0}{r_i}\tilde D_{m'}(\xi_i))$. In particular, after passing to even larger $i$ such that $t_i<\varepsilon$, we have $(\bar X_i, E_i+\delta'\varepsilon_m\frac{r_0}{r_i}\tilde D_{m'}(\xi_i))$ is also lc near $E_i$. Now by adjunction, this implies $(E_i, \delta'\varepsilon_m\frac{r_0}{r_i}D_{m'}(\xi_i))$ is lc, and hence by Lemma~\ref{lem:lctadj} we have 
    \[
    \lct_x(X_{v_0},\Delta_{v_0};D_{m'}(\xi_{v_0}))\ge \min\left\{\frac{A_{X_{v_0},\Delta_{v_0}}(\xi_{v_i})}{v_{\xi_i}(D_{m'}(\xi_{v_0}))}, \d'\varepsilon_m\right\}\ge \min\{\frac{A_{X_{v_0},\Delta_{v_0}}(\xi_{v_i})}{S_{m'}(\xi_{v_0}; \xi_{v_i})}, \d'\varepsilon_m\}.
    \]
    Since $\frac{A_{X_{v_0},\Delta_{v_0}}(\xi_{v_i})}{S_{m'}(\xi_{v_0}; \xi_{v_i})}\to \frac{A_{X_{v_0},\Delta_{v_0}}(\xi_{v_0})}{S_{m'}(\xi_{v_0}; \xi_{v_0})}$ as $i\to\infty$, after passing to even larger $i$, we may assume $\frac{A_{X_{v_0},\Delta_{v_0}}(\xi_{v_i})}{S_{m'}(\xi_{v_0}; \xi_{v_i})}\ge \varepsilon_m \frac{A_{X_{v_0},\Delta_{v_0}}(\xi_{v_0})}{S_{m'}(\xi_{v_0}; \xi_{v_0})}$.    
    Since this is true for any basis type divisors, we have shown 
    \[\delta_{m'}\ge \varepsilon_{m}\min\{\frac{A_{X_{v_0}}(\xi_{v_0})}{S_{m'}(\xi_{v_0}; \xi_{v_0})},\d'\}\] 
    and we are done by letting $m\to \infty$.
\end{proof}

%
%
\subsection{Finite generation of K-semistable valuations}
Even though K-semistable valuations are not unique at infinity, one might still hope that they are finitely generated. However, even the divisorial case seems mysterious. On the other hand, similar to the geometric setting~\cite{TY2}, if there is a Tian--Yau type compactification close enough to the K-semistable valuation, then we will be able to perturb the compactification to the K-semistable one. More precisely, the following result implies Theorem~C.
\begin{thm}\label{thm:kssfg}
    Let $v_0$ be a K-semistable negative quasimonomial valuation, and suppose there is a Tian--Yau type compactification $(\bar X, F)$ such that
    \[
    T(v_0; F)T(F; v_0)\le 1-A_X(F)^{-1}.
    \]
    Then $v_0$ is finitely generated.
\end{thm}
Before going to the proof, let us remark that the left hand side is scaling invariant and bounded below by $1$. So the quantity is measuring how close $v_0$ is to the given compactification. In fact, it is not hard to show that the expression $\log T(v_0; F)T(F; v_0)$ defines a metric closely related to the $\de_\infty$ metric on the space of bounded negative norms. On the other hand, the condition comes out of a fairly direct computation rewriting~\eqref{equ:lctcomp} as an lc condition on a pair of the form $(\bar X, \bar D+cF; \fb_\bullet)$ by making sure $c\ge 0$; thus we have chosen to not describe the condition in terms of the $\de_\infty$ metric.


\begin{proof}[Proof of Theorem~\ref{thm:kssfg}]
    We will show that in fact $v_0$ is finitely generated on $\bar X$ by realizing $v_0$ as an lc place of some special complement of $\bar X$. Let $\pi: (Y, E)\to \bar X$ be a log smooth model of $X$ such that $v_0$ lies in the interior of $\QM_\eta(Y, E)$ for $\eta=c_Y(v_0)$, and that the exceptional locus supports an ample $\Q$-divisor $H$. Let $L=\pi^*F+\varepsilon H$ for some fixed $0<\varepsilon\ll 1$ such that $L$ is ample. Pick $G\in |L|_\Q$ general such that its support does not contain any stratum of $(Y, E)$, and let $\bar D_0=\pi_*G$. Then $\bar D_0=\frac 1m \{s=0\}$ for some $s\in H^0(\bar X, mF)$, and $f\coloneqq \frac{s}{s_E^m}\in R$. Let $0<t\ll 1$ be a rational number such that $D=t\div(f)$ satisfies $v_0(D)>\frac{A_X(v_0)}{n}$. Then by Lemma~\ref{lem:ksslct}, we have
    \[
    r_0\=-A_{X, D}(v_0) = \lct(X, D; \chi_{v_0}). 
    \] 
    Equivalently, for any $v\in X^\val_\infty$, we have
    \begin{equation}\label{equ:lctcomp}
    A_{X, D}(v)\ge -r_0\FS(\chi_{v_0})(v)
    \end{equation}
    In particular, $A_{X, D}(F)\ge -r_0T(v_0; F)$. Note also that by Remark~\ref{rmk:tylct}, we have $\lct(\chi_F)=-A_X(F)$, giving $-A_X(v_0)\le -A_X(F)T(F; v_0)$. Thus by assumption, we have
    \[
    1-A_X(F)\ge -A_X(F)T(F; v_0)T(v_0; F)\ge -A_X(v_0)T(v_0;F)=r_0T(v_0; F)-v_0(D)T(v_0; F).
    \]
    Since $\ord_F(D)\ge T(v_0; F)v_0(D)$, we have $r_0T(v_0; F)<1-A_{X, D}(F)$. Set $c\=\frac{1-A_{X, D}(F)}{r_0}$, we get $v_0\le T(v_0; F)^{-1}\ord_F\le c^{-1}\ord_F$. Now using~\eqref{equ:FSbaseideal}, ~\eqref{equ:lctcomp} reads
    \[
    A_{X,D}(v)\ge -r_0(cv(F)-v(\fb_\bullet))=r_0v(\fb_\bullet)-(1-A_{X, D}(F))v(F),
    \]
    where $\fb_\bullet$ is the associated base ideals.
    Using $A_{X, D}(v) = A_{\bar X, \bar D}(v) - (1-A_{X, D}(F))v(F)$, we have
    \[
    A_{\bar X, \bar D}(v)\ge r_0v(\fb_\bullet),
    \]
    This shows that $(\bar X, \bar D; r_0\cdot \fb_\bullet)$ is an lc pair, with $v_0$ being an lc place. 
    
    If $v_0=\ord_E$ is divisorial, then we may choose $0<\varepsilon\ll1$ such that for $m\gg 0$, $0<A_{\bar X, \bar D+(1-\varepsilon)\frac{r_0}{m}\fb_m}(E)<1$. Now by~\cite{BCHM}, we may extract $E$ over $\bar X$ as the unique prime divisor, and it follows that $E$ is finitely generated on $\bar X$, and therefore also finitely generated on $X$. 
    
    In general, for $v_0$ quasimonomial, we may proceed as in~\cite[Lemma 3.1]{LXZ} to get a $\Q$-complement $\Gamma\ge \bar D$ of $\bar X$ with $v_0$ as an lc place, which is special with respect to $(Y,E)$. This shows $v_0$ is finitely generated on $\bar X$ and hence also on $X$. More precisely, for any $0<\varepsilon\ll1$, we may find divisors $F_1, \dots, F_l$ over $\bar X$, where $l$ is the rational rank of $v_0$ so that
    \[
    A_{\bar X, \bar D+r_0(1-\varepsilon_0)\fb_\bullet}(F_i)<\varepsilon,
    \]
    for $0<\varepsilon_0\ll1$. Choose $0<\delta_m<\min\{r_0, m\lct(\bar X, \bar D; \fb_m)\}$ with 
    $\frac{\delta_m}{m}\to r_0(1-\varepsilon_0)$. Then for a general element $D_m=\{s_m=0\}$ for $s_m\in \fb_m$, we have $(\bar X, \bar D+\frac{\delta_mr_0}{m}D_m)$ is log Fano, and that $A_{\bar X,\bar D+\frac{\delta_m}{m}D_m}(F_i)<\varepsilon$. The rest of the proof of Lemma~3.1 and Lemma~3.2 in~\cite{LXZ} gives a $\Q$-complement $\Gamma\ge\bar D$ of $\bar X$, with $v_0$ as an lc place. We conclude by~\cite[Theorem~4.2]{LXZ}.
\end{proof}

\begin{exam}
    Let $X=\bar X\setminus F$ where $\bar X=\mathrm{Bl}_{p}\P^3$, and $F=\mathrm{Bl}_{p}Q\in |-\frac 12 K_{\bar X}|$, where $Q$ is a quadric passing through $p$. Then $A_X(F)=-1$, and let $v_0$ be the K-semistable valuation described in~\cite{CH15}. It is not hard to check that $T(v_0; F)T(F; v_0)<2$. In fact, the two valuations have the same associated graded.
\end{exam}

%
%

\printbibliography

@article {AZ22,
    AUTHOR = {Abban, Hamid and Zhuang, Ziquan},
     TITLE = {K-stability of {F}ano varieties via admissible flags},
   JOURNAL = {Forum Math. Pi},
  FJOURNAL = {Forum of Mathematics. Pi},
    VOLUME = {10},
      YEAR = {2022},
     PAGES = {Paper No. e15, 43},
      ISSN = {2050-5086},
   MRCLASS = {14J45 (32Q20)},
  MRNUMBER = {4448177},
       DOI = {10.1017/fmp.2022.11},
       URL = {https://doi.org/10.1017/fmp.2022.11},
}

@article {BirkarNeg,
    AUTHOR = {Birkar, Caucher},
     TITLE = {Existence of log canonical flips and a special {LMMP}},
   JOURNAL = {Publ. Math. Inst. Hautes \'Etudes Sci.},
  FJOURNAL = {Publications Math\'ematiques. Institut de Hautes \'Etudes
              Scientifiques},
    VOLUME = {115},
      YEAR = {2012},
     PAGES = {325--368},
      ISSN = {0073-8301,1618-1913},
   MRCLASS = {14E30},
  MRNUMBER = {2929730},
MRREVIEWER = {Hsian-Hua\ Tseng},
       DOI = {10.1007/s10240-012-0039-5},
       URL = {https://doi.org/10.1007/s10240-012-0039-5},
}

@article {Birbdd,
    AUTHOR = {Birkar, Caucher},
     TITLE = {Anti-pluricanonical systems on {F}ano varieties},
   JOURNAL = {Ann. of Math. (2)},
  FJOURNAL = {Annals of Mathematics. Second Series},
    VOLUME = {190},
      YEAR = {2019},
    NUMBER = {2},
     PAGES = {345--463},
      ISSN = {0003-486X,1939-8980},
   MRCLASS = {14J45 (14C20 14E05 14E30)},
  MRNUMBER = {3997127},
MRREVIEWER = {James\ McKernan},
       DOI = {10.4007/annals.2019.190.2.1},
       URL = {https://doi.org/10.4007/annals.2019.190.2.1},
}

@article {BCHM,
    AUTHOR = {Birkar, Caucher and Cascini, Paolo and Hacon, Christopher D.
              and McKernan, James},
     TITLE = {Existence of minimal models for varieties of log general type},
   JOURNAL = {J. Amer. Math. Soc.},
  FJOURNAL = {Journal of the American Mathematical Society},
    VOLUME = {23},
      YEAR = {2010},
    NUMBER = {2},
     PAGES = {405--468},
      ISSN = {0894-0347,1088-6834},
   MRCLASS = {14E30 (14E05)},
  MRNUMBER = {2601039},
MRREVIEWER = {Mark\ Gross},
       DOI = {10.1090/S0894-0347-09-00649-3},
       URL = {https://doi.org/10.1090/S0894-0347-09-00649-3},
}

@article {BlumMin,
    AUTHOR = {Blum, Harold},
     TITLE = {Existence of valuations with smallest normalized volume},
   JOURNAL = {Compos. Math.},
  FJOURNAL = {Compositio Mathematica},
    VOLUME = {154},
      YEAR = {2018},
    NUMBER = {4},
     PAGES = {820--849},
      ISSN = {0010-437X,1570-5846},
   MRCLASS = {14B05 (12J20)},
  MRNUMBER = {3778195},
MRREVIEWER = {K.\ Kiyek},
       DOI = {10.1112/S0010437X17008016},
       URL = {https://doi.org/10.1112/S0010437X17008016},
}

@article {BlJ,
    AUTHOR = {Blum, Harold and Jonsson, Mattias},
     TITLE = {Thresholds, valuations, and {K}-stability},
   JOURNAL = {Adv. Math.},
  FJOURNAL = {Advances in Mathematics},
    VOLUME = {365},
      YEAR = {2020},
     PAGES = {107062, 57},
      ISSN = {0001-8708,1090-2082},
   MRCLASS = {14C20 (14M25)},
  MRNUMBER = {4067358},
MRREVIEWER = {Chenyang\ Xu},
       DOI = {10.1016/j.aim.2020.107062},
       URL = {https://doi.org/10.1016/j.aim.2020.107062},
      LABEL={BlJ},
}

@article {BE21,
    AUTHOR = {Boucksom, S\'{e}bastien and Eriksson, Dennis},
     TITLE = {Spaces of norms, determinant of cohomology and {F}ekete points
              in non-{A}rchimedean geometry},
   JOURNAL = {Adv. Math.},
  FJOURNAL = {Advances in Mathematics},
    VOLUME = {378},
      YEAR = {2021},
     PAGES = {Paper No. 107501, 124},
      ISSN = {0001-8708,1090-2082},
   MRCLASS = {32P05 (14G22 31C15)},
  MRNUMBER = {4192993},
MRREVIEWER = {Mattias\ Jonsson},
       DOI = {10.1016/j.aim.2020.107501},
       URL = {https://doi.org/10.1016/j.aim.2020.107501},
}

@article {nakstab2,
    AUTHOR = {Boucksom, S\'{e}bastien and Jonsson, Mattias},
     TITLE = {A non-{A}rchimedean approach to {K}-stability, {II}:
              {D}ivisorial stability and openness},
   JOURNAL = {J. Reine Angew. Math.},
  FJOURNAL = {Journal f\"{u}r die Reine und Angewandte Mathematik. [Crelle's
              Journal]},
    VOLUME = {805},
      YEAR = {2023},
     PAGES = {1--53},
      ISSN = {0075-4102,1435-5345},
   MRCLASS = {14G22 (14C20 14E30)},
  MRNUMBER = {4669034},
       DOI = {10.1515/crelle-2023-0062},
       URL = {https://doi.org/10.1515/crelle-2023-0062},
    LABEL = {BoJ},
}

@article{nakstab1,
    AUTHOR = {Boucksom, S\'ebastien and Jonsson, Mattias},
     TITLE = {A non-{A}rchimedean approach to {K}-stability, {I}: {M}etric
              geometry of spaces of test configurations and valuations},
   JOURNAL = {Ann. Inst. Fourier (Grenoble)},
  FJOURNAL = {Universit\'e{} de Grenoble. Annales de l'Institut Fourier},
    VOLUME = {75},
      YEAR = {2025},
    NUMBER = {2},
     PAGES = {829--927},
      ISSN = {0373-0956,1777-5310},
   MRCLASS = {32P05 (14G22 32Q26)},
  MRNUMBER = {4921366},
       DOI = {10.5802/aif.3668},
       URL = {https://doi.org/10.5802/aif.3668},
       LABEL={BoJ},
}

@article {CLCY,
    AUTHOR = {Collins, Tristan C. and Li, Yang},
     TITLE = {Complete {C}alabi-{Y}au metrics in the complement of two
              divisors},
   JOURNAL = {Duke Math. J.},
  FJOURNAL = {Duke Mathematical Journal},
    VOLUME = {173},
      YEAR = {2024},
    NUMBER = {18},
     PAGES = {3559--3604},
      ISSN = {0012-7094,1547-7398},
   MRCLASS = {32Q20 (32Q25 53C25 53C55)},
  MRNUMBER = {4850960},
       DOI = {10.1215/00127094-2024-0020},
       URL = {https://doi.org/10.1215/00127094-2024-0020},
}

@article {CSz18,
    AUTHOR = {Collins, Tristan C. and Sz\'{e}kelyhidi, G\'{a}bor},
     TITLE = {K-semistability for irregular {S}asakian manifolds},
   JOURNAL = {J. Differential Geom.},
  FJOURNAL = {Journal of Differential Geometry},
    VOLUME = {109},
      YEAR = {2018},
    NUMBER = {1},
     PAGES = {81--109},
      ISSN = {0022-040X,1945-743X},
   MRCLASS = {32Q26 (53C25 53C55)},
  MRNUMBER = {3798716},
MRREVIEWER = {Ruadha\'{\i}\ Dervan},
       DOI = {10.4310/jdg/1525399217},
       URL = {https://doi.org/10.4310/jdg/1525399217},
  LABEL={CSz},
}

@article{CH22,
  title= {Classification of asymptotically conical {C}alabi-{Y}au manifolds},
  author= {Conlon, J. Ronan and Hein, Hans-Joachim},
  journal={\url{arXiv:2201.00870}},
  year={2022},
}

@article {CH15,
    AUTHOR = {Conlon, Ronan J. and Hein, Hans-Joachim},
     TITLE = {Asymptotically conical {C}alabi-{Y}au metrics on
              quasi-projective varieties},
   JOURNAL = {Geom. Funct. Anal.},
  FJOURNAL = {Geometric and Functional Analysis},
    VOLUME = {25},
      YEAR = {2015},
    NUMBER = {2},
     PAGES = {517--552},
      ISSN = {1016-443X,1420-8970},
   MRCLASS = {53C55 (32Q25 53C25)},
  MRNUMBER = {3334234},
MRREVIEWER = {Gabriel\ Eduard\ V\^ilcu},
       DOI = {10.1007/s00039-015-0319-6},
       URL = {https://doi.org/10.1007/s00039-015-0319-6},
}

@article {CR,
    AUTHOR = {Conlon, Ronan J. and Rochon, Fr\'ed\'eric},
     TITLE = {New examples of complete {C}alabi-{Y}au metrics on {$\mathbb{C}^n$} for {$n \geq 3$}},
   JOURNAL = {Ann. Sci. \'Ec. Norm. Sup\'er. (4)},
  FJOURNAL = {Annales Scientifiques de l'\'Ecole Normale Sup\'erieure.
              Quatri\`eme S\'erie},
    VOLUME = {54},
      YEAR = {2021},
    NUMBER = {2},
     PAGES = {259--303},
      ISSN = {0012-9593,1873-2151},
   MRCLASS = {32Q25 (14J32 53C55)},
  MRNUMBER = {4258163},
       DOI = {10.24033/asens.2459},
       URL = {https://doi.org/10.24033/asens.2459},
}

@article {DS2,
    AUTHOR = {Donaldson, Simon and Sun, Song},
     TITLE = {Gromov-{H}ausdorff limits of {K}\"ahler manifolds and
              algebraic geometry, {II}},
   JOURNAL = {J. Differential Geom.},
  FJOURNAL = {Journal of Differential Geometry},
    VOLUME = {107},
      YEAR = {2017},
    NUMBER = {2},
     PAGES = {327--371},
      ISSN = {0022-040X,1945-743X},
   MRCLASS = {53C55 (32Q20 53C23)},
  MRNUMBER = {3707646},
MRREVIEWER = {Stuart\ James\ Hall},
       DOI = {10.4310/jdg/1506650422},
       URL = {https://doi.org/10.4310/jdg/1506650422},
}

@article {HMX,
    AUTHOR = {Hacon, Christopher D. and McKernan, James and Xu, Chenyang},
     TITLE = {A{CC} for log canonical thresholds},
   JOURNAL = {Ann. of Math. (2)},
  FJOURNAL = {Annals of Mathematics. Second Series},
    VOLUME = {180},
      YEAR = {2014},
    NUMBER = {2},
     PAGES = {523--571},
      ISSN = {0003-486X,1939-8980},
   MRCLASS = {14E05 (14C20 14E30)},
  MRNUMBER = {3224718},
MRREVIEWER = {Alexandr\ V.\ Pukhlikov},
       DOI = {10.4007/annals.2014.180.2.3},
       URL = {https://doi.org/10.4007/annals.2014.180.2.3},
}

@article{Huang,
    author = {Huang, Kai},
    title = {K-stability of log {F}ano cone singularities},
    journal = {Thesis (Ph.D.)–Massachusetts Institute of Technology},
    year = {2022},
}

@incollection {JonDynBerko,
    AUTHOR = {Jonsson, Mattias},
     TITLE = {Dynamics of {B}erkovich spaces in low dimensions},
 BOOKTITLE = {Berkovich spaces and applications},
    SERIES = {Lecture Notes in Math.},
    VOLUME = {2119},
     PAGES = {205--366},
 PUBLISHER = {Springer, Cham},
      YEAR = {2015},
      ISBN = {978-3-319-11028-8},
%      ISBN = {978-3-319-11028-8; 978-3-319-11029-5},
   MRCLASS = {37P50 (14G22 26E30)},
  MRNUMBER = {3330767},
MRREVIEWER = {Benjamin\ A.\ Hutz},
       DOI = {10.1007/978-3-319-11029-5\_6},
       URL = {https://doi.org/10.1007/978-3-319-11029-5_6},
}

@article {JM12,
    AUTHOR = {Jonsson, Mattias and Musta\c{t}\u{a}, Mircea},
     TITLE = {Valuations and asymptotic invariants for sequences of ideals},
   JOURNAL = {Ann. Inst. Fourier (Grenoble)},
  FJOURNAL = {Universit\'{e} de Grenoble. Annales de l'Institut Fourier},
    VOLUME = {62},
      YEAR = {2012},
    NUMBER = {6},
     PAGES = {2145--2209},
      ISSN = {0373-0956,1777-5310},
   MRCLASS = {14F18 (12J20 14B05)},
  MRNUMBER = {3060755},
MRREVIEWER = {Carlos\ Galindo},
       DOI = {10.5802/aif.2746},
       URL = {https://doi.org/10.5802/aif.2746},
}

@article{JW26,
  title= {The space of polarized compactifications of an affine variety},
  author= {Jonsson, Mattias and Wu, Yueqiao},
  journal={\url{arXiv:2610.21838}},
  year={2026},
}

@article {KY,
    AUTHOR = {Keel, Sean and Yu, Tony Yue},
     TITLE = {The {F}robenius structure theorem for affine log
              {C}alabi-{Y}au varieties containing a torus},
   JOURNAL = {Ann. of Math. (2)},
  FJOURNAL = {Annals of Mathematics. Second Series},
    VOLUME = {198},
      YEAR = {2023},
    NUMBER = {2},
     PAGES = {419--536},
      ISSN = {0003-486X,1939-8980},
   MRCLASS = {14J33 (13F60 14G22 14J32 14N35 14T15)},
  MRNUMBER = {4635299},
MRREVIEWER = {Paul\ A.\ Hacking},
       DOI = {10.4007/annals.2023.198.2.1},
       URL = {https://doi.org/10.4007/annals.2023.198.2.1},
}

@book {SingMMP,
    AUTHOR = {Koll\'ar, J\'anos},
     TITLE = {Singularities of the minimal model program},
    SERIES = {Cambridge Tracts in Mathematics},
    VOLUME = {200},
      NOTE = {With a collaboration of S\'andor Kov\'acs},
 PUBLISHER = {Cambridge University Press, Cambridge},
      YEAR = {2013},
     PAGES = {x+370},
      ISBN = {978-1-107-03534-8},
   MRCLASS = {14E30 (14B05)},
  MRNUMBER = {3057950},
MRREVIEWER = {Tommaso\ De Fernex},
       DOI = {10.1017/CBO9781139547895},
       URL = {https://doi.org/10.1017/CBO9781139547895},
}

@book {KollarMori,
    AUTHOR = {Koll\'{a}r, J\'{a}nos and Mori, Shigefumi},
     TITLE = {Birational geometry of algebraic varieties},
    SERIES = {Cambridge Tracts in Mathematics},
    VOLUME = {134},
      NOTE = {With the collaboration of C. H. Clemens and A. Corti,
              Translated from the 1998 Japanese original},
 PUBLISHER = {Cambridge University Press, Cambridge},
      YEAR = {1998},
     PAGES = {viii+254},
      ISBN = {0-521-63277-3},
   MRCLASS = {14E30},
  MRNUMBER = {1658959},
MRREVIEWER = {Mark\ Gross},
       DOI = {10.1017/CBO9780511662560},
       URL = {https://doi.org/10.1017/CBO9780511662560},
}

@article {Linvol,
    AUTHOR = {Li, Chi},
     TITLE = {Minimizing normalized volumes of valuations},
   JOURNAL = {Math. Z.},
  FJOURNAL = {Mathematische Zeitschrift},
    VOLUME = {289},
      YEAR = {2018},
    NUMBER = {1-2},
     PAGES = {491--513},
      ISSN = {0025-5874,1432-1823},
   MRCLASS = {14B05},
  MRNUMBER = {3803800},
MRREVIEWER = {Chen\ Jiang},
       DOI = {10.1007/s00209-017-1963-3},
       URL = {https://doi.org/10.1007/s00209-017-1963-3},
}

@article {LWX,
    AUTHOR = {Li, Chi and Wang, Xiaowei and Xu, Chenyang},
     TITLE = {Algebraicity of the metric tangent cones and equivariant
              {K}-stability},
   JOURNAL = {J. Amer. Math. Soc.},
  FJOURNAL = {Journal of the American Mathematical Society},
    VOLUME = {34},
      YEAR = {2021},
    NUMBER = {4},
     PAGES = {1175--1214},
      ISSN = {0894-0347,1088-6834},
   MRCLASS = {14B07 (14E30 14J17 14J45 53C55)},
  MRNUMBER = {4301561},
MRREVIEWER = {Carlos\ Galindo},
       DOI = {10.1090/jams/974},
       URL = {https://doi.org/10.1090/jams/974},
}

@article{LX18,
    AUTHOR = {Li, Chi and Xu, Chenyang},
     TITLE = {Stability of valuations: higher rational rank},
   JOURNAL = {Peking Math. J.},
  FJOURNAL = {Peking Mathematical Journal},
    VOLUME = {1},
      YEAR = {2018},
    NUMBER = {1},
     PAGES = {1--79},
      ISSN = {2096-6075},
   MRCLASS = {14E30 (14B05 32Q26)},
  MRNUMBER = {4059992},
MRREVIEWER = {Jingjun Han},
       DOI = {10.1007/s42543-018-0001-7},
       URL = {https://doi.org/10.1007/s42543-018-0001-7},
}

@article {LXKcomp,
    AUTHOR = {Li, Chi and Xu, Chenyang},
     TITLE = {Stability of valuations and {K}oll\'ar components},
   JOURNAL = {J. Eur. Math. Soc. (JEMS)},
  FJOURNAL = {Journal of the European Mathematical Society (JEMS)},
    VOLUME = {22},
      YEAR = {2020},
    NUMBER = {8},
     PAGES = {2573--2627},
      ISSN = {1435-9855,1435-9863},
   MRCLASS = {14J17 (13A18 14E20)},
  MRNUMBER = {4118616},
MRREVIEWER = {Jos\'e\ Mar\'ia\ Tornero},
       DOI = {10.4171/JEMS/972},
       URL = {https://doi.org/10.4171/JEMS/972},
}

@article {GLiu,
    AUTHOR = {Liu, Gang},
     TITLE = {Compactification of certain {K}\"ahler manifolds with
              nonnegative {R}icci curvature},
   JOURNAL = {Adv. Math.},
  FJOURNAL = {Advances in Mathematics},
    VOLUME = {382},
      YEAR = {2021},
     PAGES = {Paper No. 107652, 27},
      ISSN = {0001-8708,1090-2082},
   MRCLASS = {53C55 (32Q05 32Q15 53C21)},
  MRNUMBER = {4224050},
MRREVIEWER = {Leonardo\ Biliotti},
       DOI = {10.1016/j.aim.2021.107652},
       URL = {https://doi.org/10.1016/j.aim.2021.107652},
}

@article{LW,
      title={K-semistability of log Fano cone singularities}, 
      author={Yuchen Liu and Yueqiao Wu},
      year={2025},
      journal={\url{arXiv:2408.05189}},
}

@article {LXZ,
    AUTHOR = {Liu, Yuchen and Xu, Chenyang and Zhuang, Ziquan},
     TITLE = {Finite generation for valuations computing stability
              thresholds and applications to {K}-stability},
   JOURNAL = {Ann. of Math. (2)},
  FJOURNAL = {Annals of Mathematics. Second Series},
    VOLUME = {196},
      YEAR = {2022},
    NUMBER = {2},
     PAGES = {507--566},
      ISSN = {0003-486X,1939-8980},
   MRCLASS = {14J45 (14D20 14E30 32Q20)},
  MRNUMBER = {4445441},
MRREVIEWER = {James\ McKernan},
       DOI = {10.4007/annals.2022.196.2.2},
       URL = {https://doi.org/10.4007/annals.2022.196.2.2},
}

@article {OX,
    AUTHOR = {Odaka, Yuji and Xu, Chenyang},
     TITLE = {Log-canonical models of singular pairs and its applications},
   JOURNAL = {Math. Res. Lett.},
  FJOURNAL = {Mathematical Research Letters},
    VOLUME = {19},
      YEAR = {2012},
    NUMBER = {2},
     PAGES = {325--334},
      ISSN = {1073-2780,1945-001X},
   MRCLASS = {14E30},
  MRNUMBER = {2955764},
MRREVIEWER = {James\ McKernan},
       DOI = {10.4310/MRL.2012.v19.n2.a5},
       URL = {https://doi.org/10.4310/MRL.2012.v19.n2.a5},
}

@article {Sunbubble,
    AUTHOR = {Sun, Song},
     TITLE = {Bubbling of {K}\"ahler-{E}instein metrics},
   JOURNAL = {Pure Appl. Math. Q.},
  FJOURNAL = {Pure and Applied Mathematics Quarterly},
    VOLUME = {21},
      YEAR = {2025},
    NUMBER = {3},
     PAGES = {1317--1348},
      ISSN = {1558-8599,1558-8602},
   MRCLASS = {32Q20 (14E30)},
  MRNUMBER = {4852036},
       DOI = {10.4310/pamq.250116024725},
       URL = {https://doi.org/10.4310/pamq.250116024725},
}

@article {SunZhang,
    AUTHOR = {Sun, Song and Zhang, Junsheng},
     TITLE = {No semistability at infinity for {C}alabi-{Y}au metrics
              asymptotic to cones},
   JOURNAL = {Invent. Math.},
  FJOURNAL = {Inventiones Mathematicae},
    VOLUME = {233},
      YEAR = {2023},
    NUMBER = {1},
     PAGES = {461--494},
      ISSN = {0020-9910,1432-1297},
   MRCLASS = {53C55 (32Q26)},
  MRNUMBER = {4602001},
MRREVIEWER = {Antonella\ Nannicini},
       DOI = {10.1007/s00222-023-01187-4},
       URL = {https://doi.org/10.1007/s00222-023-01187-4},
}

@article {GaborCY,
    AUTHOR = {Sz{\'{e}}kelyhidi, G\'{a}bor},
     TITLE = {Degenerations of {$\C^n$} and {C}alabi-{Y}au metrics},
   JOURNAL = {Duke Math. J.},
  FJOURNAL = {Duke Mathematical Journal},
    VOLUME = {168},
      YEAR = {2019},
    NUMBER = {14},
     PAGES = {2651--2700},
      ISSN = {0012-7094,1547-7398},
   MRCLASS = {53C55 (53C25)},
  MRNUMBER = {4012345},
MRREVIEWER = {Anna\ M.\ Fino},
       DOI = {10.1215/00127094-2019-0021},
       URL = {https://doi.org/10.1215/00127094-2019-0021},
}

@article {Szeuniq,
    AUTHOR = {Sz\'ekelyhidi, G\'abor},
     TITLE = {Uniqueness of some {C}alabi-{Y}au metrics on {${\bf C}^n$}},
   JOURNAL = {Geom. Funct. Anal.},
  FJOURNAL = {Geometric and Functional Analysis},
    VOLUME = {30},
      YEAR = {2020},
    NUMBER = {4},
     PAGES = {1152--1182},
      ISSN = {1016-443X,1420-8970},
   MRCLASS = {32Q25 (53C55)},
  MRNUMBER = {4153912},
MRREVIEWER = {Yalong\ Shi},
       DOI = {10.1007/s00039-020-00543-3},
       URL = {https://doi.org/10.1007/s00039-020-00543-3},
}

@article {TY1,
    AUTHOR = {Tian, G. and Yau, Shing-Tung},
     TITLE = {Complete {K}\"ahler manifolds with zero {R}icci curvature.
              {I}},
   JOURNAL = {J. Amer. Math. Soc.},
  FJOURNAL = {Journal of the American Mathematical Society},
    VOLUME = {3},
      YEAR = {1990},
    NUMBER = {3},
     PAGES = {579--609},
      ISSN = {0894-0347,1088-6834},
   MRCLASS = {53C55 (32C10 53C25)},
  MRNUMBER = {1040196},
MRREVIEWER = {Alan\ Michael\ Nadel},
       DOI = {10.2307/1990928},
       URL = {https://doi.org/10.2307/1990928},
}

@article {TY2,
    AUTHOR = {Tian, Gang and Yau, Shing-Tung},
     TITLE = {Complete {K}\"ahler manifolds with zero {R}icci curvature.
              {II}},
   JOURNAL = {Invent. Math.},
  FJOURNAL = {Inventiones Mathematicae},
    VOLUME = {106},
      YEAR = {1991},
    NUMBER = {1},
     PAGES = {27--60},
      ISSN = {0020-9910,1432-1297},
   MRCLASS = {32C17 (32F07 53C21 53C55)},
  MRNUMBER = {1123371},
MRREVIEWER = {Chi-Keung\ Cheung},
       DOI = {10.1007/BF01243902},
       URL = {https://doi.org/10.1007/BF01243902},
}

@article {NALFC1,
    AUTHOR = {Wu, Yueqiao},
     TITLE = {Volume and {M}onge-{A}mp\`ere energy on polarized affine
              varieties},
   JOURNAL = {Math. Z.},
  FJOURNAL = {Mathematische Zeitschrift},
    VOLUME = {301},
      YEAR = {2022},
    NUMBER = {1},
     PAGES = {781--809},
      ISSN = {0025-5874,1432-1823},
   MRCLASS = {32W20 (14R20 32U40)},
  MRNUMBER = {4405668},
MRREVIEWER = {Vincent\ Guedj},
       DOI = {10.1007/s00209-021-02935-z},
       URL = {https://doi.org/10.1007/s00209-021-02935-z},
}

@article {Xuqm,
    AUTHOR = {Xu, Chenyang},
     TITLE = {A minimizing valuation is quasi-monomial},
   JOURNAL = {Ann. of Math. (2)},
  FJOURNAL = {Annals of Mathematics. Second Series},
    VOLUME = {191},
      YEAR = {2020},
    NUMBER = {3},
     PAGES = {1003--1030},
      ISSN = {0003-486X,1939-8980},
   MRCLASS = {14E30 (14J17 14J45)},
  MRNUMBER = {4088355},
MRREVIEWER = {Yuchen\ Liu},
       DOI = {10.4007/annals.2020.191.3.6},
       URL = {https://doi.org/10.4007/annals.2020.191.3.6},
}

@book {XuKstab,
    AUTHOR = {Xu, Chenyang},
     TITLE = {K-stability of {F}ano varieties},
    SERIES = {New Mathematical Monographs},
    VOLUME = {50},
 PUBLISHER = {Cambridge University Press, Cambridge},
      YEAR = {2025},
     PAGES = {xi+411},
      ISBN = {978-1-009-53877-0; [9781009538763]},
   MRCLASS = {14J45 (32Q26)},
  MRNUMBER = {4893062},
}

@article {XZ21,
    AUTHOR = {Xu, Chenyang and Zhuang, Ziquan},
     TITLE = {Uniqueness of the minimizer of the normalized volume function},
   JOURNAL = {Camb. J. Math.},
  FJOURNAL = {Cambridge Journal of Mathematics},
    VOLUME = {9},
      YEAR = {2021},
    NUMBER = {1},
     PAGES = {149--176},
      ISSN = {2168-0930,2168-0949},
   MRCLASS = {14B05 (13A18 14E30)},
  MRNUMBER = {4325260},
MRREVIEWER = {Yuchen\ Liu},
       DOI = {10.4310/CJM.2021.v9.n1.a2},
       URL = {https://doi.org/10.4310/CJM.2021.v9.n1.a2},
}

@article {XZstab,
    AUTHOR = {Xu, Chenyang and Zhuang, Ziquan},
     TITLE = {Stable degenerations of singularities},
   JOURNAL = {J. Amer. Math. Soc.},
  FJOURNAL = {Journal of the American Mathematical Society},
    VOLUME = {38},
      YEAR = {2025},
    NUMBER = {3},
     PAGES = {585--626},
      ISSN = {0894-0347,1088-6834},
   MRCLASS = {14J17 (13A18 14B05 14E30 32Q26)},
  MRNUMBER = {4890655},
       DOI = {10.1090/jams/1055},
       URL = {https://doi.org/10.1090/jams/1055},
}

@article{SZ26,
  title= {On complete Ricci--flat Kähler manifolds},
  author= {Sun, Song and Zhang, Junsheng},
  journal={\url{arXiv:2610.29891}},
  year={2026},
}

@article{Liu26,
  title= {Special valuations and automorphisms of affine log Calabi--Yau varieties},
  author= {Liu, Yuchen},
  journal={\url{arXiv:2609.22737}},
  year={2026},
}

@article {LX16,
    AUTHOR = {Li, Chi and Xu, Chenyang},
     TITLE = {Stability of valuations and {K}oll\'ar components},
   JOURNAL = {J. Eur. Math. Soc. (JEMS)},
  FJOURNAL = {Journal of the European Mathematical Society (JEMS)},
    VOLUME = {22},
      YEAR = {2020},
    NUMBER = {8},
     PAGES = {2573--2627},
      ISSN = {1435-9855,1435-9863},
   MRCLASS = {14J17 (13A18 14E20)},
  MRNUMBER = {4118616},
MRREVIEWER = {Jos\'e\ Mar\'ia\ Tornero},
       DOI = {10.4171/JEMS/972},
       URL = {https://doi.org/10.4171/JEMS/972},
}

\end{document}